\documentclass[11pt,reqno,a4paper]{amsart}

\usepackage[T1]{fontenc}
\usepackage[utf8]{inputenc}

\usepackage{amsmath}
\usepackage{algorithm}
\usepackage{amssymb}
\usepackage{graphicx}
\RequirePackage{enumitem}
\setenumerate{itemsep=0pt}
\setitemize{itemsep=0pt}

\usepackage{cases}
\usepackage{subfigure}
\usepackage{verbatim}

\usepackage{comment}

\usepackage{amsthm}
\usepackage{thmtools}

\usepackage{url}

\newcommand{\eps}{\varepsilon}

\renewcommand{\phi}{\varphi}
\renewcommand{\bar}[1]{\overline{#1}}

\renewcommand{\leq}{\leqslant}
\renewcommand{\geq}{\geqslant}

\newcommand{\Prob}{\mathcal{P}}

\newcommand{\Sph}{\mathcal{S}}
\newcommand{\spt}{\mathrm{spt}}
\newcommand{\B}{\mathrm{B}}
\newcommand{\D}{\mathrm{D}}

\newcommand{\sca}[2]{\langle #1 | #2\rangle}

\DeclareMathOperator{\argmin}{argmin}

\newcommand{\Wass}{\operatorname{W}}

\newcommand{\nr}[1]{\left\Vert #1\right\Vert}

\renewcommand{\d}{\mathrm{d}}
\newcommand{\dd}{\mathrm{d}}

\newcommand{\LL}{\mathrm{L}}

\newcommand{\Class}{\mathcal{C}}
\newcommand{\Lip}{\mathrm{Lip}}

\DeclareMathOperator{\diam}{diam}

\newcommand{\proj}{\mathrm{proj}}

\newcommand{\Rsp}{\mathbb{R}}

\usepackage{xspace}

\newcommand{\R}{\mathbb{R}}
\newcommand{\TC}{\mathcal{T}}

\newcommand{\X}{\mathcal{X}}
\newcommand{\Y}{\mathcal{Y}}
\newcommand{\Ix}{\mathfrak{I}_x}

\newcommand{\Grhok}{\mathcal{G}_K \circ \overrightarrow{\rho_K}}
\newcommand{\K}{\mathcal{K}}

\usepackage{color}

\DeclareMathOperator{\cexp}{c-exp}
\DeclareMathOperator{\inter}{int}
\DeclareMathOperator{\Dom}{Dom}
\newcommand{\Dompx}{\Dom'(\nabla_x c)}
\newcommand{\Dompy}{\Dom'(\nabla_y c)}

\newtheorem{theorem}{Theorem}
\newtheorem{lemma}[theorem]{Lemma}
\newtheorem{remark}[theorem]{Remark}
\newtheorem{proposition}[theorem]{Proposition}
\newtheorem{corollary}[theorem]{Corollary}
\newtheorem{definition}[theorem]{Definition}

\title{Strong c-concavity and stability in optimal transport}
\author{Anatole Gallou\"{e}t}
\address{Universit\'e  Grenoble Alpes, CNRS, Grenoble INP, LJK, 38000
Grenoble, France}
\author{Quentin M\'erigot}
\address{Université Paris-Saclay, CNRS, Laboratoire de mathématiques d’Orsay, 91405, Orsay, France}
\author{Boris Thibert}
\address{Universit\'e  Grenoble Alpes, CNRS, Grenoble INP, LJK, 38000
Grenoble, France}
\date{\today}

\begin{document}
\maketitle

\begin{abstract}
  The stability of solutions to optimal transport problems under variation of the measures is fundamental from a mathematical viewpoint: it is closely related to the convergence of numerical approaches to solve optimal transport problems and justifies many of the applications of optimal transport. 
In this article, we introduce the notion of strong $c$-concavity, and we show that it plays an important role for proving stability results in optimal transport for general cost functions $c$. 
We then introduce a differential criterion for proving that a function is strongly $c$-concave, under an hypothesis on the cost introduced originally by Ma-Trudinger-Wang for establishing regularity of optimal transport maps. Finally, we provide two examples where this stability result can be applied, for cost functions taking value $+\infty$ on the sphere: the reflector problem and the Gaussian curvature measure prescription problem.
\end{abstract}

\tableofcontents
\section{Introduction}
The theory of optimal transport has had an important impact in applied
mathematics, with applications in inverse problems, in variational
modeling of evolution PDEs
\cite{villani2003topics,santambrogio2015optimal}, and in machine
learning \cite{peyre2019computational} to name but a few. Numerical
applications of this theory have been made possible thanks to the
tremendous progress of optimal transport solvers in the last decade
\cite{peyre2019computational, merigot2021optimal, benamou2021optimal}.

The stability of solutions to optimal transport problems under
variation of the data is fundamental from a mathematical viewpoint,
making optimal transport a ``well-posed'' problem in the terminology
of Hadamard. The question of \emph{quantitative stability} is also of
prime importance. The first and most obvious reason is that it is
strongly related to the convergence of many numerical approaches to
solve optimal transport problems --- both in statistical and in numerical
analysis contexts --- and explicitly or implicitly it justifies most of
the  applications of optimal transport.  Quantitative
stability is at the heart of several other applications, including the
understanding of geometric embeddings of spaces of probability
measures to Hilbert spaces used in statistics~\cite{delalande2021quantitative},
the convergence analysis of numerical methods for evolution equations
using optimal transport as a building block~\cite{bourne2022semi},
the estimation of transport maps in high dimension~\cite{hutter2021minimax}
or the construction of precise asymptotics for random matching problems~\cite{ambrosio2019optimal}.

The stability of optimal transport plans can be established in a very
general setting \cite{villani2008optimal}, under variations of the
source and target measures, and even under variations of the
cost. However, the question of \emph{quantitative} stability has only
been addressed rather recently, and most of the existing results deal
with the cost function $c(x,y) = \nr{x-y}^2$
\cite{gigli2011holder,berman2021convergence,
  delalande2021quantitative,li2021quantitative}, or with the squared
geodesic distance on a Riemannian manifold \cite{ambrosio2019optimal}.
The aim of this article is to establish stability results for more
general cost functions, namely those that satisfy the \emph{strong
Twist} and \emph{Ma-Trudinger-Wang} conditions on manifolds. We also
identify \emph{strong $c$-concavity} of the Kantorovitch potential as
a central notion to get stability results.

\subsection*{Optimal transport}
Let $M,N$ be two Polish spaces, let $\mu \in \Prob(M)$, $\nu \in
\Prob(N)$ be two probability measures on $M$ and $N$ and let
$c:M\times N\to \R \cup \{+\infty\}$ be a lower semi-continuous cost function which is bounded below. A
\emph{transport map} between $\mu$ and $\nu$ is a map $T:M\to N$ such that the image measure $T_\#\mu$ equals $\nu$. 
Monge's optimal transport problem between $\mu$ and $\nu$ for the
cost $c$ amounts to finding a map $T:M \to N$ that minimizes
\begin{equation}
\label{eq:MP}
\tag{MP}
\inf_{T_\# \mu = \nu} \int_M c(x, T(x))d\mu(x).
\end{equation}
Such a map, if it exists, is called an \emph{optimal transport map}
between $\mu$ and $\nu$.  Existence and uniqueness of such an optimal
transport map is obtained for instance when the transport cost is
quadratic, i.e.  $c(x,y) =\nr{x-y}^2$, when $M,N$ are compact subsets
of $\Rsp^d$, and when $\mu$ is absolutely continuous with respect to
the Lebesgue measure \cite{brenier1991polar}. Existence and uniqueness
also hold for more general cost functions satisfying a so-called
``twist'' hypothesis \cite{gangbo1995optimal}.

Kantorovich's relaxation consists in minimizing the same quantity, but
among \emph{transport plans} $\Gamma(\mu,\nu)$:
\begin{equation}
\label{eq:KP}
\tag{KP}
\min_{\gamma \in \Gamma(\mu,\nu)} \int_{M\times N} c(x,y) \dd\gamma(x,y).
\end{equation}
We recall that a
transport plan between $\mu$ and $\nu$ is a probability measure
$\gamma\in\Prob(M\times N)$ with marginals $\mu$ and $\nu$. 
Under mild assumptions (e.g. $c$ is lower-semicontinuous and bounded below), a minimizer
to \eqref{eq:KP} always exists -- but uniqueness may fail. A minimizer
to \eqref{eq:KP} is called an optimal transport plan.

\subsection*{Existing stability results}
The problem of stability of optimal transport maps can be expressed as
a continuity property of the map $(\mu,\nu) \mapsto T_{\mu\to\nu}$,
where $T_{\mu\to\nu}$ is the optimal transport map between a source
probability measure $\mu$ and a target measure $\nu$. In order to have
a common space in which to consider the optimal transport map
$T_{\mu\to\nu}$, we will mainly consider the problem of the stability of
the map $T_{\nu} := T_{\mu\to\nu}$ for a fixed $\mu$. As first noted by Li and Nochetto
\cite{li2021quantitative}, the arguments implying quantitative
stability of $\nu \mapsto T_{\mu\to\nu}$ sometimes also imply general
stability results, where both the source and target measures can
change.

To the best of our knowledge, the first quantitative stability result
in optimal transport is of ``local'' nature, in the sense that it only
holds near a configuration $(\mu,\nu)$, and is established under
strong assumptions on the data. It is due to Ambrosio and reported in
an article of Gigli \cite{gigli2011holder}.  It can be phrased as
follows.

\begin{theorem}[Ambrosio-Gigli]
  Assume that $M$ and $N$ are compact subsets of $\Rsp^d$,  that
  $\mu \in\Prob(M)$ is absolutely continuous, and that for some
  $\nu_0\in \Prob(N)$ the optimal transport map $T_{\mu \to \nu_0}$
  for the quadratic cost $c(x,y) = \nr{x-y}^2$ is Lipschitz. Then
\begin{equation} \label{eq:AmbrosioGigli}
  \forall \nu_1\in\Prob(N), ~~\nr{T_{\mu\to \nu_0} - T_{\mu\to\nu_1}}^2_{L^2(\mu)} \leq \diam(M) \Lip(T_{\mu\to \nu_0}) W_1(\nu_0, \nu_1).
\end{equation}
\end{theorem}
In the above statement, $\Lip(T)$ is the Lipschitz constant of
 the map $T$ and $\Wass_1(\nu_0,\nu_1)$ is the Wasserstein distance between
$\nu_0$ and $\nu_1$ with respect to the Euclidean distance on $N$.

By Brenier theorem \cite{brenier1991polar}, we know that $T_\nu =
\nabla \phi_\nu$, where $\phi_\nu$ is convex.  A convex analysis
result shows that the Lipschitz regularity of $T_\nu$ is equivalent to
the strong convexity of the convex conjugate $\psi_\nu =
\phi_\nu^*$. Using these remarks, the proof of the stability estimate
\eqref{eq:AmbrosioGigli} can then be obtained in a few lines, see e.g.
\cite[Theorem 2.2]{delalande2021quantitative}. Li and
Nochetto~\cite{li2021quantitative} prove under the same hypothesis
that 
if $\gamma \in\Prob(M\times N)$ is the
transport plan between $\mu$ and $\nu$ induced by the optimal map $T_{\mu\to\nu}$, and 
$\tilde{\gamma}$ is \emph{any} optimal transport plan between $\tilde{\mu}$
and $\tilde{\nu}$, i.e. any solution to \eqref{eq:KP} then
\[ \Wass_2(\gamma, \tilde{\gamma})^2 \leq C (\Wass_2(\mu, \tilde{\mu}) + \Wass_2(\nu, \tilde{\nu})) ,\]
where $C$ is a constant that depends on $\Lip(T_{\mu\to\nu})$, the
diameters of $M$ and $N$. The Wasserstein distance $\Wass_2$ in the
left-hand side is with respect to a product metric on $M\times N$.

We mention that the ``Euclidean'' stability result
\eqref{eq:AmbrosioGigli} can be extended to optimal transport problems
on a compact Riemannian manifold with the squared geodesic distance~\cite{ambrosio2019optimal}.  We also mention the
more ``global'' stability results of \cite{berman2021convergence,
  delalande2021quantitative}, which do not make regularity assumptions on $T_{\mu\to\nu}$, but come with worse
continuity estimates. For instance, the main theorem of
\cite{delalande2021quantitative} shows that if $\mu \in\Prob(\Rsp^d)$
is a probability density on a compact convex subset of $\Rsp^d$, which
is bounded from above and below by a positive constant, then for any
compact subset $Y\subseteq\Rsp^d$, the map $\nu \mapsto T_{\mu\to\nu}$
is $\frac{1}{6}$-Hölder from $(\Prob(Y),\Wass_1)$ to
$\LL^2(\mu,\Rsp^d)$, to be compared to the $\frac12$ exponent in
\eqref{eq:AmbrosioGigli}.




\subsection*{Strong $c$-concavity of the potential}

A key ingredient in the stability results for the quadratic cost~\cite{gigli2011holder, ambrosio2019optimal} is the strong convexity of the Kantorovich potentials $\psi$ associated to the optimal transport maps. In order to get stability results for general cost functions $c$, we introduce below the notion of \emph{strong $c$-concavity}. 

We denote by $\d_N : N  \times N \to \R_+$ the distance on $N$. The $p$-Wasserstein distance on $\Prob(N)$ between two probability measures is defined with respect to the  distance by
\[ W_p^p(\nu_0,\nu_1) = \inf_{\gamma \in \Gamma(\nu_0, \nu_1) } \int_{N \times N} d_N(y,z)^p d\gamma(y,z),\]

\begin{definition}[Transport map induced by a potential]
Let $T : M \to N$ be a measurable map, and $\psi: N \to \R$. We say that $T$ is induced by $\psi$, or that $\psi$ is a potential associated to $T$ if
\[ \forall x \in M, \quad T(x) \in \argmin_{y \in N} c(x,y) - \psi(y) \]
\end{definition}
Thanks to Kantorovich duality~\cite{villani2003topics}, we know that if a transport map $T$ from $\mu$ to $\nu$ is induced by a potential $\psi$ then T is a solution to the Monge problem~\eqref{eq:MP}.
Such a potential $\psi$ can be constructed by solving the dual problem
\begin{equation}\tag{DP}
\label{eq:KD}
\sup_{\psi : N \to \R} \int_M \psi^c \dd \mu + \int_N \psi \dd \nu
\end{equation}
where $\psi^c : M \to \R$ is the c-transform of $\psi$, defined by
\[ \psi^c (x) = \inf_{y \in N} c(x,y) - \psi(y) \]
so that $\psi^c(x) + \psi(y) \leq c(x,y)$. The dual problem~\eqref{eq:KD} has a maximizer, for instance, if the cost $c$ is continuous on the compact $M \times N$, but existence also holds with weaker hypothesis on $c$, see~\cite{villani2008optimal} for instance.
When such a maximizer exists, and still by Kantorovich theory, we can assume that a map $T$ solution of~\eqref{eq:MP} is induced by a $c$-concave potential $\psi$. We recall the notion of $c$-concavity, and we refer to~\cite{villani2008optimal}.
\begin{definition}[$c$-concavity and $c$-conjugate]
We say that $\psi : N \to \R \cup \{ - \infty \}$ is c-concave if for any $y \in N$ there exists $x \in M$ such that
\[\forall z \in N, \quad c(x,z) - \psi(z) \geq c(x,y) - \psi(y) \]
An equivalent definition is that there exists a function $\phi : M \to \R \cup \{ \pm \infty \}$ such that for any $y \in N$
\[ \psi(y) = \inf_{x \in M} c(x,y) - \phi(x). \]
We denote the right-hand side of the above equation by $\phi^c(x)$, and we call it the $c$-conjugate of $\phi$.
One can define similarly the notion of $c$-concave function on $M$.
\end{definition}
The $c$-superdifferential of $\psi$ at a point $y \in N$ is defined by
\begin{equation}
    \label{c-superdiff}
    \partial^c \psi(y) = \{x \in M \mid \forall z \in N, \psi(z) - c(x,z) \leq \psi(y) - c(x,y)  \}
\end{equation}
Note that $\psi$ is $c$-concave iff for any $y \in N$ its c-superdifferential $ \partial^c \psi(y)$ is non-empty. We can now introduce the notion of strong c-concavity.

\begin{definition}[strong $c$-concavity on $D$]
\label{def:stcconc}
We say that a c-concave function $\psi$ is strongly $c$-concave on a set $D \subseteq M \times N$ and with modulus $\omega$ if for all $x,y,z$ such that $(x,y) \in D, (x,z) \in D$ and $x \in \partial^c \psi(y)$:
\begin{equation}
\label{strong_concavity}
\psi(z) - c(x,z) \leq \psi(y) - c(x, y) - \omega(\d_N(y,z))
\end{equation}
\end{definition}
In the above definition, the modulus $\omega : \R_+ \to \R_+$ is an increasing function that satisfies $\omega(0)=0$.
One can check that when $c(x,y) = - \sca{x}{y}$ and $\omega(r) = C r^2$ the notion of strong concavity and strong c-concavity are equivalent.
Moreover if a function $\psi : N \to \R$ is strongly c-concave, then for $y \neq z$ in $N$, $\partial^c \psi(y) \cap \partial^c \psi(z) = \emptyset$, or equivalently for $x \in M$ there exists a unique minimizer of $y \mapsto c(x,y) - \psi(y)$.
This implies that the transport map associated to $\psi$ is uniquely defined by minimizing $c(x,\cdot) -\psi$:
\[ \forall x \in M \quad T(x) = \argmin_{y \in N} c(x,y) - \psi(y) \]

\subsection*{Contribution} This paper is concerned with stability problems in optimal transport. We introduce the notion of \emph{strong c-concavity}, which is central  to get stability results.

\begin{itemize}
  \item We provide two stability results  in Section~\ref{sec:stability} that depend on an assumption of strong $c$-concavity. First, we extend the $1/2$-H\"older  stability result of Ambrosio stated in~\cite{gigli2011holder} to general cost function $c$ (Theorem~\ref{th:stability-cconc}). Our result is local around transport maps associated to \emph{strongly $c$-concave} potential. Second, we generalize a result of  Li and Nochetto~\cite{li2021quantitative} that estimates the distance of a transport plan to an optimal transport map (the source and target measures being fixed) in terms of the suboptimality gap  (Proposition~\ref{prop:stabplan}). 
We then use this result to obtain quantitative stability  of the transport plan with respect to both measures (Proposition~\ref{prop:stabbothmeasureW1}), following the strategy of Li-Nochetto ~\cite{li2021quantitative} for the quadratic cost.

\item We provide in Section~\ref{sec:criterioncconc} the central result of this paper (Theorem~\ref{th:criterionstrongcconc}), which is a differential criterion for a potential function $\psi$ to be strongly $c$-concave. This result generalizes a sufficient condition for c-convexity proposed by Villani~\cite[Th. 12.46]{villani2008optimal}. It requires
 that $M, N$ are two smooth $d$-dimensional complete Riemannian manifolds.
Similarly to Villani, we require a local condition on the derivatives of the potential $\psi$ and a weak  Ma-Trudinger-Wang condition~\cite{Ma2005regularity} . 
In Section~\ref{sec:otstability}, we combine Theorem~\ref{th:criterionstrongcconc} to the stability results of Section~\ref{sec:stability} to get  local stability results for optimal  transport maps.

\item The last two sections are dedicated to the applications of our stability results to two optimal transport problems on the sphere, with cost functions taking the value $+\infty$.
In Section~\ref{sec:reflector} we consider the reflector antenna problem, which is a non-imaging optics problem that can be written as optimal transport~\cite{wang2004design}.
Section~\ref{sec:gaussmeasure} is dedicated to the prescription of the Gaussian curvature measure of a convex body, originally introduced by Alexandrov~\cite{alexandrov1950convex} and rephrased as an optimal transport problem by Oliker~\cite{oliker2007embedding}.
\end{itemize}

\section{Stability under strong c-concavity}\label{sec:stability}
In this section we assume that $M$ and $N$ are Polish spaces.  We provide stability results in the neighborhood of transport maps that are associated to strongly c-concave Kantorovitch potential. The stability result of Section~\ref{sec:stabtarget} is  with respect to variations of the target measure, whereas the result in Section~\ref{sec:stabboth} is with respect to variations of both the source and the target measures. This last result is a consequence of an error bound for a fixed optimal transport problem given in Section~\ref{subsec:errorbound}. As a side note, we also remark in the last section that strong $c$-concavity implies H\"{o}lder regularity of transport maps. 

\subsection{Stability with respect to the target measure}\label{sec:stabtarget}
The following theorem extends to general cost functions a theorem of Ambrosio~ \cite{gigli2011holder}, using a reformulation proposed in \cite{delalande2021quantitative}. 
The hypothesis that the transport map $T$ is Lipschitz (in the formulation of \cite{delalande2021quantitative}) is replaced by the assumption that the transport map is induced by a strongly c-concave potential $\psi$, i.e.
$$
\forall x \in M\quad  T(x) \in \argmin_{y\in N} c(x,y) - \psi(y).
$$
\begin{theorem}
\label{th:stability-cconc}
Let $D \subseteq M \times N$ be a compact set and $c : M \times N \to \R \cup \{ + \infty\}$ be a cost function of class $\Class^1$ on $D$.
Let $\mu \in \Prob(M)$ and $\nu_0,\nu_1 \in \Prob(N)$.  We assume that there exists optimal transport maps $T_i$ from $\mu$ to $\nu_i$ with associated potential $\psi_i : N \to \R$ ($i=0,1$) such that:
\begin{itemize}
    \item $\psi_0$ is Lipschitz on $N$ and c-concave on $D$.
    \item $\psi_1$ is Lipschitz on $N$ and strongly c-concave with modulus $\omega$ on $D$.
    \item The maps $T_i$ satisfies for any $x \in M$, $(x,T_i(x)) \in D$.
\end{itemize}
Then,
\begin{equation}
\label{eq:stabtarget}
\int_M \omega(\d_N(T_0 (x), T_1(x))) d\mu(x) \leq (\Lip(\psi_0) + \Lip(\psi_1)) W_1(\nu_0, \nu_1) 
\end{equation}
\end{theorem}

\begin{remark}
\label{rem:dist}
The left hand side of inequality~\eqref{eq:stabtarget} measures the  distance between transport maps $T_0$ and $T_1$. To see this let us consider a simpler case where $M$ and $N$ are domains of $\R^d$ and $\omega(r) = r^2$ then we get
\[\int_M \omega(\d_N(T_0 (x), T_1(x))) d\mu(x) = \nr{T_1 - T_0}_{L^2(\mu)}^2 \]
and in that case, Theorem~\ref{th:stability-cconc} amounts to bounding the $L^2$ norm of the distance between transport maps.
\end{remark}

\begin{remark}[Discretization of the target measure]
Assume that we have two absolutely continuous measures $\mu \in \Prob(M)$ and $\nu \in \Prob(N)$ and an optimal transport map $T$ from $\mu$ to $\nu$ satisfying all the hypothesis of Theorem~\ref{th:stability-cconc}.
One can pick a family of points $(y_i)_{1 \leq i \leq n}$ in the target space $N$ and approximate the measure $\nu$ by a discrete measure $\nu_h$ of the form
\[ \nu_h = \sum_i \nu(V_i) \delta_{y_i} \]
where $(V_i)_{1 \leq i \leq n}$ is a Voronoi tesselation of $N$ around the points $(y_i)_{1 \leq i \leq n}$ chosen in an appropriate way in the support of $\nu$.
The parameter $h$ is given by $h = \max_{1 \leq i \leq n} \diam(V_i)$ so that $W_1(\nu, \nu_h) \leq h$. 
We can compute the optimal transport map $T_h$ between $\mu$ and $\nu_h$ using semi-discrete methods such as~\cite{kitagawa2019convergence}. Then, Theorem~\ref{th:stability-cconc} implies
\[ \int_M \omega(\d_N(T(x), T_h(x))) d\mu(x) \leq C h\]
where the constant $C$ depends on the Lipschitz constants of the potentials, which can be controlled explicitely in many cases.
If the modulus $\omega(r)$ is quadratic, then the $\LL^2(\mu)$ distance between $T$ and $T_h$ is controlled by $h^{1/2}$.
\end{remark}

\begin{proof}[Proof of Theorem~\ref{th:stability-cconc}]
We have
\[\sca{\nu_1 - \nu_0}{\psi_1 - \psi_0} = \int_N \psi_1 d_N(\nu_1 - \nu_0) +  \int_N \psi_0 d_N(\nu_0 - \nu_1) \]
Let $A = \int_N \psi_1 d_N(\nu_1 - \nu_0)$ and $B = \int_N \psi_0 d_N(\nu_0 - \nu_1)$.
Since $T_{i\#} \mu = \nu_i$ we have 
\begin{align*}
A &= \int_N \psi_1 d\nu_1 - \int_N \psi_1 d\nu_0 \\
 &= \int_M \psi_1(T_1(x)) d\mu(x) - \int_M \psi_1(T_0(x)) d\mu(x)
\end{align*}
For $x \in M$ we have $x \in \partial^c \psi_i(T_i(x))$. Then the strong $c$-concavity of $\psi_1$ gives
\begin{align*} 
A &= \int_M \psi_1(T_1(x)) - \psi_1(T_0(x)) d\mu(x) \\
&\geq \int_M c(x, T_1(x)) - c(x, T_0(x)) + \omega(\d_N(T_0(x), T_1(x))) d\mu 
\end{align*}
Now since $\psi_0$ is also $c$-concave, we have
\[ B \geq \int_M - c(x, T_1(x)) + c(x, T_0(x)) d\mu\]
Summing these two inequalities gives
\[ \int_M \omega(\d_N(T_0 (x), T_1(x))) d\mu(x) \leq  \int_N \psi_1 - \psi_0 d_N(\nu_1 - \nu_0) \]
Since $\psi_0$ and $\psi_1$ are Lipschitz, we have 
\begin{align*}
\int_N \psi_1 - \psi_0 d_N(\nu_1 - \nu_0) 
&\leq (\Lip(\psi_0) + \Lip(\psi_1)) W_1(\nu_0, \nu_1)
\end{align*}
where the last inequality is given by Kantorovich-Rubinstein theorem.
\end{proof}

\subsection{Error bounds for optimal transport problems}\label{subsec:errorbound}
In this section, we generalize in Proposition~\ref{prop:stabplan} a stability result of  Li and Nochetto~\cite{li2021quantitative} to general cost functions, using the notion of strong $c$-concavity. This result allows to bound in Corollary~\ref{cor:stabplanW1} the Wasserstein distance between the optimal transport map and any transport plan with the same marginals by the suboptimality gap of the transport plan.

\begin{proposition}
\label{prop:stabplan}
Let $\mu \in \Prob(M)$, $\nu \in \Prob(N)$ and $T: M \to N$ be an optimal transport map from $\mu$ to $\nu$.
We assume that $T$
is induced by a strongly c-concave potential $\psi : N \to \R$  with modulus $\omega$ on a compact subset  $D$ of $M\times N$ wich contains the graph of $T$. 
Then any transport plan $\gamma \in \Gamma(\mu, \nu)$ supported on $D$ satisfies
\begin{equation*}
\int_{M \times N} \omega(\d_N(T(x),y)) d\gamma(x,y) \leq \int_{M \times N} c(x,y) d\gamma(x,y) - \int_M c(x,T(x)) d\mu(x)
\end{equation*}
\end{proposition}
The left hand side of this equation is called the suboptimality gap of $\gamma$, and measures how worse the transport plan $\gamma$ behaves compared to the optimal transport map $T$.
\begin{proof}
The strong c-concavity of $\psi$ implies that for any $x,y \in D$,
\[ \psi(y) \leq \psi(T(x)) - c(x,T(x)) + c(x,y) - \omega(\d_N(T(x),y)). \]
Moreover since $T_\#\mu = \nu$, we have
\[ \int_N \psi(y)d\nu(y) = \int_M \psi(T(x)) d\mu(x) \]
which combined with the strong c-concavity of $\psi$ gives
\begin{align*}
0 &= \int_N \psi(y)d\nu(y) - \int_M \psi(T(x)) d\mu(x)\\
&= \int_{D} \psi(y) - \psi(T(x)) d\gamma(x,y) \\
&\leq \int_{D} c(x,y) - c(x,T(x)) - \omega(\d_N(T(x),y)) d\gamma(x,y) \\
&= \int_{D} c(x,y)d\gamma(x,y) - \int_M c(x,T(x))d\mu(x) - \int_{D} \omega(\d_N(T(x), y))d\gamma(x,y)
\end{align*}
Rearranging this inequality gives the desired conclusion.
\end{proof}

We can rephrase this proposition using the the 1-Wasserstein distance
$\Wass_1$ in $\Prob(M \times N)$ induced by the distance
$$\d_{M \times
  N}((x,y),(x',y')) = \d_M(x,x') + \d_N(y,y').$$

\begin{corollary}
\label{cor:stabplanW1}
Under the assumptions of Proposition~\ref{prop:stabplan}, if the modulus of the Kantorovitch potential $\psi$ is $\omega(r) = Cr^2$, one has
\[ \Wass_1(\gamma, \gamma_T) \leq \frac{1}{\sqrt{C}} \left( \int_{M \times N} c(x,y) d\gamma(x,y) - \int_M c(x,T(x)) d\mu(x) \right)^{1/2} \]
where $\gamma_T = (Id, T)_\# \mu$.
\end{corollary}
\begin{proof}
Let $S : M \times N \to (M \times N)^2$ defined by
\[ S(x,y) = (S_1(x,y), S_2(x,y))\]
where $S_1(x,y) = (x,T(x))$ and $S_2(x,y) = (x,y)$.
Let $\pi = S_\# \gamma \in \Prob((M \times N)^2)$. One can check that $\pi \in \Gamma(\gamma_T, \gamma)$, which implies
\begin{align*}
W_1(\gamma_T, \gamma) &\leq \int_{(M \times N)^2} \d_{M \times N}((x,y),(x',y')) \dd \pi(x,y,x',y') \\
&= \int_{M \times N} \d_{M \times N}(S_1(x,y), S_2(x,y)) \dd \gamma(x,y) \\
&= \int_{M \times N} \d_N(T(x),y) \dd \gamma(x,y).
\end{align*}
We use the Cauchy-Schwarz inequality in $L^2(M \times N, \gamma)$ and
Proposition~\ref{prop:stabplan} to get the desired result.
\end{proof}

\subsection{Stability with respect to both measures}\label{sec:stabboth}

Here we apply Corollary~\ref{cor:stabplanW1}  to show stability results of transport plans with respect to both the source and the target measures.  Our result holds for general cost functions and is inspired by a result of Li and Nochetto~\cite{li2021quantitative} that holds in the quadratic case. 
We denote by $\d_M$ the distance on $M$ and $\d_N$ the distance on $N$.
We also choose for distance on the product space $\d_{M \times N}((x,y),(x',y')) = \d_M(x,x') + \d_N(y,y')$.
Throughout this section, we require the cost function $c$ to be Lipschitz on the whole product space $M \times N$.

\begin{proposition}[Stability with respect to both measures]
\label{prop:stabbothmeasureW1} Let $\mu , \tilde{\mu} \in \Prob(M)$ and $\nu , \tilde{\nu} \in \Prob(N)$. 
Let $c : M \times N \to \R$ be a cost function which is Lipschitz on $M \times N$.
Let $T : M \to N$ be an optimal transport map between $\mu$ and $\nu$, and $\tilde{\gamma}$ be an optimal transport plan between $\tilde{\mu}$ and $\tilde{\nu}$ for the cost $c$. We assume that $T$ is induced by a strongly c-concave potential  $\psi : N \to \R$ with associated modulus $\omega(r) = C r^2$ on $D=M \times N$.
Then we have
\[ \Wass_1(\gamma_T, \tilde{\gamma}) \leq \eps + \sqrt{\frac{2 \Lip(c)}{C} \eps}, \quad \hbox{where } \eps := \Wass_1(\tilde{\mu},\mu) + \Wass_1(\nu,\tilde{\nu}).  \]
\end{proposition}

The end of this section is devoted to the proof of this proposition.
As in \cite{li2021quantitative}, we will use the gluing lemma
~\cite{santambrogio2015optimal,villani2008optimal}.
\begin{lemma}[gluing of measures]
\label{lemma:gluing}
Let $(X_i, \mu_i)$ be probability spaces for $i \in \{1,2,3\}$, and $\gamma_{12} \in \Gamma(\mu_1, \mu_2)$, $\gamma_{23} \in \Gamma(\mu_2, \mu_3)$.
Then there exists $\pi \in \Prob(X_1 \times X_2 \times X_3)$ such that $\pi(\cdot, \cdot, X_3) = \gamma_{12}$ and $\pi(X_1, \cdot, \cdot) = \gamma_{23}$.
Or equivalently 
\[ p_{12\#}\pi = \gamma_{12} \quad p_{23\#}\pi = \gamma_{23} \]
 where $p_{ij}$ is the projection defined by $p_{ij}(x_1, x_2, x_3) = (x_i, x_j)$.
\end{lemma}
We also need the following (easy) lemma, showing that the transport cost
$$\TC^c(\mu,\nu) := \min_{\gamma\in\Gamma(\mu,\nu)} \int c\dd\gamma$$
is Lipschitz with respect to perturbations of the measures when $c$ is Lipschitz.
\begin{lemma}
\label{lemma:globalcostgap}
Let $c : M \times N \to \R$ be a Lipschitz cost function.
Let $\mu, \tilde{\mu} \in \Prob(M)$ and $\nu, \tilde{\nu} \in \Prob(N)$. 
Then we have
\[ \left| \TC^c(\mu, \nu) - \TC^c(\tilde{\mu},\tilde{\nu}) \right|
\leq \Lip(c) (\Wass_1(\mu, \tilde{\mu}) + \Wass_1(\nu, \tilde{\nu})).\]
\end{lemma}
\begin{proof} Kantorovich duality gives
  \[ \TC^c(\mu, \nu) = \max_{\phi \oplus \psi \leq c} \int_{M} \phi \dd \mu + \int_N \psi \d \nu. \]
  Moreover, since the cost is Lipschitz, the maximum is attained in
  the dual problem; one can assume that the maximum is attained for
  two potentials $\phi,\psi$ satisfying $\phi =\psi^c$ and $\phi =
  \psi^c$. In particular both $\phi$ and $\psi$ are Lipschitz
  continuous with Lipschitz constant lower than $\Lip(c)$.  Kantorovitch (weak)
  duality applied to the two measures $\tilde{\mu}$ and $\tilde{\nu}$
  gives
\[ \TC^c(\tilde{\mu},\tilde{\nu}) \geq \int_M \phi \d \tilde{\mu}+ \int_N \psi \d \tilde{\nu}. \]
We thus get
\[\TC^c(\mu, \nu) - \TC^c(\tilde{\mu},\tilde{\nu})  \leq \int_M \phi \d (\mu -\tilde{\mu})+ \int_N \psi \d (\nu - \tilde{\nu}) \leq \Lip(c) ( \Wass_1(\mu, \nu) + \Wass_1(\tilde{\mu},\tilde{\nu}) )\]
where the last inequality is given by Kantorovich-Rubinstein Theorem.
By symmetry the same result holds when we exchange $\mu, \nu$ and $\tilde{\mu},\tilde{\nu}$.
\end{proof}

\begin{proof}[Proof of Proposition~\ref{prop:stabbothmeasureW1}]
Let $\alpha \in \Gamma(\mu, \tilde{\mu})$ and $\beta \in \Gamma(\tilde{\nu},\nu)$ be optimal transport plans for the cost $d_M$ and $d_N$.
Let $\pi \in \Prob(M^2 \times N^2)$ be a gluing of $\alpha, \tilde{\gamma}$ and $\beta$, i.e. 
\[ p_{12\#}\pi = \alpha, \quad p_{23\#}\pi = \tilde{\gamma}, \quad p_{34\#}\pi = \beta \]
Defining  $\gamma = p_{14\#} \pi \in \Gamma(\mu, \nu)$, we get
\begin{align}
\label{ineq:W1plansmarginals}
\Wass_1(\gamma, \tilde{\gamma}) 
&\leq \int_{M^2 \times N^2} \d_M(x,x') + \d_N(y,y') \d \pi(x,x',y,y') \nonumber \\
&= \int_{M^2} \d_M(x,x') \d \alpha(x,x') + \int_{N^2} \d_N(y,y') \d \beta(y,y') \nonumber  \\
&= \Wass_1(\tilde{\mu},\mu) + \Wass_1(\nu,\tilde{\nu})
\end{align}
We also have 
\begin{align}
\label{ineq:gammatilde}
& \int_{M \times N} c(x,y) \dd \gamma \nonumber \\
= &\int_{M^2 \times N^2} c(x,y) \dd \pi(x,x',y',y) \nonumber \\
= &\int_{M^2 \times N^2} c(x',y') + c(x, y) - c(x',y') \dd \pi(x,x',y',y) \nonumber \\ 
\leq &\int_{M^2 \times N^2} c(x',y') + \Lip(c)(\d_M(x,x') + \d_N(y,y')) \dd \pi(x,x',y',y) \nonumber \\ 
= &\int_{M \times N} c(x',y') \dd \tilde{\gamma} + \Lip(c)\left( \int_{M^2} \d_M(x,x')\dd\alpha +  \int_{N^2} \d_N(y,y') \dd\beta \right) \nonumber \\ 
\leq &\int_{M \times N} c(x,y)\dd \tilde{\gamma} + \Lip(c) (W_1(\mu, \tilde{\mu}) + W_1(\nu, \tilde{\nu}))
\end{align}
The transport plans $\gamma_T = (Id, T)_\# \mu \in \Gamma(\mu, \nu)$ and $\tilde{\gamma} \in \Gamma(\mu, \nu)$ are  optimal, so that by Lemma~\ref{lemma:globalcostgap},
\[\int_{M \times N} c(x,y)\dd \tilde{\gamma} \leq \int_{M \times N} c(x,y)\dd \gamma_T + \Lip(c) (W_1(\mu, \tilde{\mu}) + W_1(\nu, \tilde{\nu})) \]
which combined with~\eqref{ineq:gammatilde} gives
\[\int_{M \times N} c(x,y) \dd \gamma - \int_{M \times N} c(x,y)\dd \gamma_T \leq 2 \Lip(c) (W_1(\mu, \tilde{\mu}) + W_1(\nu, \tilde{\nu})) \]
Corollary~\ref{cor:stabplanW1} then implies that
\[ \Wass_1(\gamma, \gamma_T) \leq \left[ \frac{2 \Lip(c)}{C} (\Wass_1(\tilde{\mu},\mu) + \Wass_1(\nu,\tilde{\nu}) ) \right]^{1/2} \]
Finally, using the triangle inequality along with \eqref{ineq:W1plansmarginals} we get
\begin{align*} 
\Wass_1(\tilde{\gamma}, \gamma_T) &\leq  \Wass_1(\tilde{\gamma}, \gamma) + \Wass_1(\gamma, \gamma_T) \\
&\leq \Wass_1(\tilde{\mu},\mu) + \Wass_1(\nu,\tilde{\nu}) + \left( \frac{2 \Lip(c)}{C}  (\Wass_1(\tilde{\mu},\mu) + \Wass_1(\nu,\tilde{\nu}) ) \right)^{1/2}\\
&=  \eps + \sqrt{\frac{2 \Lip(c)}{C}  \eps} \qedhere
\end{align*}

\end{proof}

\subsection{A remark on regularity}
The above results show that the notion of strong $c$-concavity is sufficient to get stability results. In fact, this notion can also lead to regularity of the associated transport maps, as expressed in the following lemma. 
\begin{lemma}[Regularity under strong $c$-concavity]
Let us assume that the cost function $c : M \times N \to \R$ is Lipschitz on $M \times N$ and let  $T : M \to N$ be a transport map induced by a  strongly c-concave potential $\psi : N \to \R$,  with continuity modulus $\omega(r) = C r^2$ on $M \times N$. Then $T$ is $1/2$-H\"older: 
\[ \d_N(T(x),T(x')) \leq \left( \frac{\Lip(c)}{C} \d_M(x,x') \right)^{1/2} \] 
\end{lemma}
\begin{proof}
Let $x \in M$. Since $T$ is induced by a strongly c-concave potential $\psi$ we have
$T(x)= \argmin_{y \in N} c(x,y) - \psi(y)$. 
The strong c-concavity of $\psi$ implies that for every $y\in N$ 
\[ c(x,y) - \psi(y) \geq c(x,T(x)) - \psi(T(x)) + \omega(\d_N(y,T(x)))\]
Now let $x' \in M$. By choosing $y = T(x')$ the above inequality becomes
\[ c(x,T(x')) - \psi(T(x')) \geq c(x,T(x)) - \psi(T(x)) + \omega(\d_N(T(x'),T(x)))\]
This inequality still holds when we exchange $x$ and $x'$, summing the two gives
\[ 2 \omega(\d_N(T(x),T(x'))) \leq c(x',T(x)) + c(x,T(x')) - c(x',T(x')) - c(x, T(x)) \]
and since $c$ Lipschitz we have
\[ C \d_N(T(x),T(x'))^2 \leq \Lip(c) \d_M(x,x'). \qedhere \] 
\end{proof}
Thus, strong $c$-concavity of the potential entails some regularity of
the transport map, generalizing what is well-known in the convex
setting (i.e. if $\psi$ is strongly convex, then $\psi^*$ is
$\Class^{1,1}$). The next section will show a partial converse
statement, under strong assumptions on the cost function.

\section{Sufficient condition for strong c-concavity}\label{sec:criterioncconc}
This section is all about the notion of strong c-concavity that we used through the previous section to deduce stability results of optimal transport maps.
From now on, we assume  that $M$ and $N$ are smooth complete Riemannian manifolds. 

It is known that the notions of convexity and strong convexity can be easily characterized by conditions on the Hessian for smooth functions.
The c-convexity is not that easy to study but for cost functions $c$ that are regular enough in a certain sense, there is a differential criterion for c-convexity, given by Villani~\cite{villani2008optimal}.
In this section we extend Villani's result for strong c-concavity, in other words we show that the strong c-concavity of a function can also be guaranteed by conditions on its derivatives.
This result is presented in Corollary~\ref{cor:strongcconc}.
To do so we need the cost function $c : M \times N \to \Rsp\cup\{+\infty\}$ to satisfy the Ma-Trudinger-Wang (MTW) condition, which is a well known condition in regularity theory of optimal transport.

\subsection{The Ma-Trudinger-Wang tensor} We recall in this section the notion of MTW tensor~\cite{villani2008optimal}. 
Recall that we are working with two smooth complete Riemannian manifolds $M$ and $N$, and a cost function $c:M\times N\to \R \cup \{+ \infty \}$.
We denote by  $\Dom(\nabla_x c) \subseteq M \times N$ the set of differentiability of the cost $c$ and $\Dom'(\nabla_x c(x, \cdot)) = \inter(\Dom(\nabla_x c(x, \cdot)))$ the interior of the domain of definition of $\nabla_x c(x, \cdot)$, then
\begin{equation}
\label{Domp}
\Dompx = \{ (x,y) \mid x \in \inter(M), y \in \Dom'(\nabla_x c(x, \cdot))) \}
\end{equation}

\begin{definition}[Twisted cost]
The cost $c$ satifies the (Twist) condition if $\nabla_x c (x, \cdot)$ is injective on its domain of definition, i.e. for any $x, y, y'$ such that $(x,y) \in \Dompx$ and $(x, y') \in \Dompx$:
\[\nabla_xc(x,y)= \nabla_xc(x,y')  \implies y = y' \]
\end{definition}

\begin{definition}[STwist]
\label{STwist}
The cost satisfies the strong Twist condition (STwist) if $c$ is $\Class^2$, $\nabla_x c$ is one-to-one and $D_{xy}^2 c$ is non singular on $\Dompx$.
\end{definition}

If the cost function satisfies (Twist), then for $x \in \inter(M)$ the function $- \nabla_x c(x, \cdot)$ is invertible on its image $\Ix \subseteq T_x M$, i.e. 
\[ - \nabla_x c(x, \cdot) : \Dom'(\nabla_x c(x, \cdot)) \subseteq N \to \Ix \subseteq T_xM \]
is one-to-one.

\begin{definition}[c-exponential]
When the cost $c$ satisfies the (Twist) condition, we can define the c-exponential for $x \in M$ by $\cexp_x = \left(- \nabla_x c (x, \cdot)\right)^{-1}$, giving for $p \in \Ix$:
\begin{align*}
 \cexp_x(p) : \Ix \subseteq T_xM  &\to \Dom'(\nabla_x c(x, \cdot))\subseteq N\\
  p &\to \big(\nabla_x c(x, \cdot)\big)^{-1}(-p)
 \end{align*}
\end{definition}

\begin{definition}[c-segment]
A c-segment is the image of a usual segment by the map $\cexp_x$.
We denote $(y_t)_{0 \leq t \leq 1} = [y_0, y_1]_{x}$ the $c$-segment between $y_0$ and $y_1$ with base $x$ defined for $p_0 = (- \nabla_x c)^{-1}(x, y_0)$ and $p_1 = (- \nabla_x c)^{-1}(x, y_1)$ by
\[ y_t = \cexp_{x}((1 - t) p_0 + t p_1)\]

\end{definition}

\begin{definition}[c-convex set] Let $A \subseteq N$. 
\begin{itemize}
\item We say that $A$ is c-convex with respect to $x \in M$ if for any $y_0, y_1 \in A$, there is a c-segment $[y_0, y_1]_x$ entirely contained in $A$.
\item The set  $A$ is said to be c-convex with respect to a set $B \subseteq M$ if $A$ is c-convex with respect to any $x \in B$.
\item A set $D \subseteq M \times N$ is said to be totally c-convex if for any two points $(x, y_0) \in D$ and $(x, y_1) \in D$, the c-segment $(y_t)_{0\leq t \leq 1}=[y_0, y_1]_x$  satisfies for any $t$ $(x,y_t)\in D$.
\item We say that $D \subseteq M \times N$ is symmetrically c-convex if both $[x_0, x_1]_y \subseteq D$ and $[y_0, y_1]_x \subseteq D$.
\end{itemize}
\end{definition}

\begin{definition}[MTW tensor]
Assuming that $c$ is of class $\Class^4$ on $ \Dompx$ and satisfies the (STwist) condition, the Ma-Trudinger-Wang tensor is defined for $(x_0,y_0) \in \Dompx$ and $(\zeta,\eta) \in T_x M \times T_y N$ by
\[ \mathfrak{S}_c (x_0,y_0)(\eta,\zeta) = -\frac{3}{2}
 \frac{\partial^2}{\partial q_{\tilde{\eta}}^2} \frac{\partial^2}{\partial y_{\zeta}^2}\big(c(\cexp_{y_0}(q),y)\big)\Big|_{y=y_0,q=-\nabla_yc(x_0,y_0)}\]
with $\tilde{\eta} = - \nabla_{xy}^2 c(x_0,y_0)  \eta \in T_xM$.
\end{definition}
In the above definition $- \nabla_{xy}^2 c(x_0,y_0) : T_xM \times T_yN \to \R$ is a bilinear form which is non singular since (STwist) is satisfied. We then identify for $\eta \in T_yN$ the linear form  $- \nabla_{xy}^2 c(x_0,y_0) \eta = \tilde{\eta} : T_xM \to \R$ with a vector of $T_xM$ using the Riemannian structure.
\begin{definition}[weak MTW]
We say that the weak MTW condition (MTWw) is satisfied on a compact set $D \subseteq M \times N$if there exists a constant $C > 0$ such that for any $(x,y) \in D$ and $(\zeta,\eta) \in T_x M \times T_y N$ we have
\begin{equation}\label{eq:MTWw}
 \mathfrak{S}_c (x,y)(\eta,\zeta) \geq - C |\sca{\zeta}{\tilde{\eta}}|\nr{\zeta}\nr{\tilde{\eta}} \tag{MTWw}
\end{equation}
This condition was introduced by Ma, Trudinger and Wang~\cite{Ma2005regularity} and is often referred to as \emph{(A3w)}.
\end{definition}

\subsection{Differential criterion for strong c-concavity}

The goal here is to generalize Villani's differential criterion~\cite{villani2008optimal} (detailed in the following theorem) for c-convexity to our definition of strong c-concavity. 
Our proof is highly inspired from Villani's one, in particular we study the same real valued function $h : [0,1] \to \R$ and show inequalities that are similar and also require positivity of the MTW tensor.

\begin{theorem}[Differential criterion for c-convexity, {\cite[Th. 12.46]{villani2008optimal}}]
\label{th:criterioncconc}
Let $D \subseteq M \times N$ be a closed symmetrically c-convex set and $c \in \Class^4(D,\R)$ such that $c$ and $\check{c}$ satisfy (STwist) on $D$. Assume that the weak MTW condition \eqref{eq:MTWw} is satisfied on $D$. Let $\X = \proj_M(D)$ and $\psi \in \Class^2(\X,\R)$. If for any $x \in \X$ there exists $y \in N$ such that $(x,y) \in D$ and
\[ \begin{cases}
\nabla\psi(x) + \nabla_xc(x,y) = 0 \\
D^2 \psi(x) + D_{xx}^2 c(x,y) \geq 0 
\end{cases} \]
Then $\psi$ is $c$-convex on $D$.
\end{theorem}
This theorem is given for a potential function $\psi$ on $\X \subseteq M$ and gives a c-convexity result while we consider $\psi : N \to \R$ and work on c-concavity, but this is really just a matter of convention.
Also Villani needs the Hessian $D^2 \psi(x) + D_{xx}^2 c(x,y)$ to be positive semi-definite to obtain c-convexity, while we are naturally going to need the Hessian $D^2_{yy}c(x,y) - D^2\psi(y)$ to have eigenvalues bounded from below by a positive constant to obtain strong c-concavity.
A noticeable difference of c-convexity with respect to convexity is that it cannot be expressed locally, as we require the MTW tensor to be positive on the whole set $D$ which is a global condition.

\begin{theorem}[Differential criterion for strong c-concavity]
\label{th:criterionstrongcconc}
We consider $D \subseteq \Dompx \cap  \Dompy$ a symmetrically c-convex compact set and denote $\X = \proj_M(D)$, $\Y = \proj_N(D)$. We assume that $c \in \Class^4(D, \R)$, that $c$ and $\check{c}$ satisfy (STwist) on $D$ where $\check{c}(x,y) = \check{c}(y,x)$.
We also assume that the weak MTW condition is satisfied on $D$.
Let $\psi \in \Class^2(\Y, \R)$ be a c-concave function on $D$ and such that there exists $\lambda > 0$ satisfying for any $x \in \partial^c\psi(y)$
\[ D^2 \psi(y) - D_{yy}^2 c(x,y) \geq \lambda Id \]
Then $\psi$ is strongly $c$-concave on $D$ with modulus $\omega(\d_N(y,z)) = C \d_N(y,z)^2$, where $C > 0$ is a constant depending on $c$, $\X$ and $\Y$. This means that we have
\[ \psi(z) - c(x,z) \geq \psi(y) - c(x,y) + C\d_N(y,z)^2 \]
for the points $x \in \X$, $y,z \in \Y$ such that $x \in \partial^c\psi(y)$, $(x,y) \in D$ and $(x,z) \in D$.
\end{theorem}

\begin{corollary}[Strong c-concavity]
\label{cor:strongcconc}
We make the same hypothesis on $c$ and $D$, and just assume $\psi \in \Class^2(\Y, \R)$.
Let $T : \X \to \Y$ the map defined by
$T(x) = \argmin_y c(x, y) - \psi(y)$ be of class $\Class^1$ and 
satisfying for any $x \in \X$, $(x, T(x)) \in D$.
Then the function $\psi$ is strongly c-concave on the set $D$ with modulus $\omega(\d_N(y,z)) = C \d_N(y,z)^2$.
\end{corollary}
\begin{remark}[Restriction of c-concavity to $D$]
Theorem~\ref{th:criterionstrongcconc} actually gives the strong $c$-concavity of the potential $\psi$ on a set $D$ where the cost function is smooth enough.
This can be an issue if we want to find a transport map $T : M \to N$ such that $T(x) = \argmin_{z \in N} c(x,y) - \psi(y)$, because we cannot be sure that the argmin will be obtained at a point $y$ such that $(x,y) \in D$.
This issue has to be treated independently for each application.
\end{remark}

\subsection{Proof of Theorem~\ref{th:criterionstrongcconc}.}
In this subsection we consider that all the hypothesis of Theorem~\ref{th:criterionstrongcconc} are satisfied.
For any $x\in \X$, we denote $\Y^x = \{ y \in N \mid (x,y) \in D \}$.
Let $\bar{y} \in \Y$ and $\bar{x} \in \partial^c \psi(\bar{y})$ such that $(\bar{x},\bar{y}) \in D$. Note that $\bar{x}$ always exists by hypothesis. Let us fix $y \in \Y^{\bar{x}}$. 
We want to show that  there exists a constant $C > 0$ independant of $\bar{x}, \bar{y}$ and $y$ such that
\begin{equation}
\label{eq:strong_cconc}
c(\bar{x},y) - \psi(y) \geq c(\bar{x}, \bar{y}) - \psi(\bar{y}) + C \d_N(y,\bar{y})^2
\end{equation}
We  put $(y_t)_{0 \leq t \leq 1} = [\bar{y}, y]_{\bar{x}}$ the $c$-segment between $\bar{y}$ and $y$ with base $\bar{x}$. Remark that the $c$-convexity of $D$ implies that for any $t$ in $[0,1]$, $(\bar{x}, y_t) \in D$.
We define the function $h$ by
\[ h(t) :=  c(\bar{x}, y_t) - \psi(y_t)\]
such that Equation~\eqref{eq:strong_cconc} writes 
\begin{equation}
\label{eq:hcconc} 
h(1) \geq h(0) + C \d_N(y,\bar{y})^2 
\end{equation}
The end of this section is devoted to the proof of Equation~\eqref{eq:hcconc}. \\

\noindent \textbf{Notation.}  We first introduce some notations. Note that $A_{\bar{x}}:=\nabla^2_{xy}c(\bar{x},y_t):T_{\bar{x}}M\times T_{y_t}N \to \Rsp$ is a bilinear form which is assumed to be nonsingular. For any $X\in T_{\bar{x}}M$ and $Y\in T_{y_t}N$, we can write $\nabla^2_{xy}c(\bar{x},y_t)(X,Y) = \sca{A_{\bar{x}} X}{Y}=\sca{^tA_{\bar{x}} Y}{X}$ where in some local coordinates $A_{\bar{x}}$ is an invertible matrix and $X$ and $Y$ are column matrices. Then $\nabla^2_{xy}c(\bar{x},y_t)(X,\cdot)$ is a linear form on $T_{y_t}N$  which is identified to the vector $A_{\bar{x}}X\in T_{y_t}N$. Similarly $^tA_{\bar{x}}Y\in T_{\bar{x}}M$. We take the same notation for $A_{x_s}=\nabla^2_{xy}c(x_s,y_t)$.

\begin{lemma}
\label{lem:hderivative}
\[ h'(t) =  \sca{\zeta}{\hat{\eta}}\]
and
\[
h''(t) = \Big(\D^2_{yy}c(x^t,y_t) - D^2\psi(y_t)\Big) (\hat{\eta},\hat{\eta})
+ \frac{2}{3} \int_0^1 \mathfrak{S}_c\Big( \cexp_{y_t}(\bar{q_t}+s\zeta), y_t \Big) (\bar{\zeta}, \hat{\eta}) (1-s) \d s,
\]
where $x^t \in \partial^c \psi(y_t)$,
\[
\begin{array}{ll}
\eta = \nabla_xc(\bar{x},\bar{y}) -\nabla_xc(\bar{x},y) \in T_{\bar{x}}M
& \hat{\eta} = -^tA_{\bar{x}}^{-1}\eta \in T_{y_t}N\\
\zeta= \nabla_yc(\bar{x},y_t) - \nabla \psi(y_t)\in T_{y_t}N
& \hat{\zeta} = -A_{\bar{x}}^{-1}\zeta \in T_{\bar{x}}M\\
\bar{q_t}:=-\nabla_yc(\bar{x},y_t) \in T_{y_t}M & \bar{\zeta} = - A_{x_s}^{-1}\zeta \in T_{x_s}N
\end{array}
\]
\end{lemma}
\begin{proof}[Proof of Lemma~\ref{lem:hderivative}]
Since $D$ is symmetrically c-convex and $(\bar{x}, \bar{y}) \in D$, $(\bar{x}, y) \in D$, we can differentiate $h$ as follows
$$
h'(t) = \sca{\nabla_yc(\bar{x},y_t) - \nabla \psi(y_t)}{\dot{y_t}}
$$
 We also have by differentiating $-\nabla_xc(\bar{x},y_t) = \bar{p} + t \eta$:
$$
\eta = -\nabla_{xy}^2c(\bar{x},y_t)\dot{y_t}= -^tA_{\bar{x}}\dot{y_t}
$$
So that $\hat{\eta}=-^tA_{\bar{x}}^{-1}\eta=\dot{y_t}$ and thus
$$
h'(t) = \sca{\zeta}{\hat{\eta}}.
$$
Differentiating $h'$ gives
$$
h''(t) = \Big(\nabla^2_{yy}c(\bar{x},y_t) - \nabla^2 \psi(y_t) \Big)(\dot{y_t},\dot{y_t}) + \sca{\zeta}{\ddot{y_t}}.
$$
By differentiating $-\eta = \nabla_{xy}^2c(\bar{x},y_t)\dot{y_t}$, one gets
$$
\nabla_{xyy}c(\bar{x},y_t)(\dot{y_t},\dot{y_t})+ \sca{\nabla_{xy}c(\bar{x},y_t)}{\ddot{y_t}}=0
$$
so that
$$
\ddot{y_t} = -^tA_{\bar{x}}^{-1} \nabla_{xyy}c(\bar{x},y_t)(\hat{\eta},\hat{\eta})
$$
and
$$
\sca{\zeta}{\ddot{y_t}} = \sca{\zeta}{-^tA_{\bar{x}}^{-1} \nabla_{xyy}c(\bar{x},y_t)(\hat{\eta},\hat{\eta})} = \sca{-A_{\bar{x}}^{-1}\zeta}{ \nabla_{xyy}c(\bar{x},y_t)(\hat{\eta},\hat{\eta})}.
$$
We therefore have
$$
h''(t)=  \Big(\nabla^2_{yy}c(\bar{x},y_t) - \nabla^2 \psi(y_t) \Big) (\hat{\eta},\hat{\eta})  + \sca{\hat{\zeta}}{\nabla_{xyy}c(\bar{x},y_t)(\hat{\eta},\hat{\eta})}
$$
We define $\Phi(x):= \Big(\nabla^2_{yy}c(x,y_t) - \nabla^2 \psi(y_t) \Big) (\hat{\eta},\hat{\eta})$.
Then we have for $X\in T_xM$
$$
D\Phi(x).X = \sca{X}{\nabla^3_{xyy}c(x,y_t)(\hat{\eta},\hat{\eta})},
$$
so that
$$
h''(t) = \Phi(\bar{x}) + D\Phi(\bar{x})\hat{\zeta}
$$
We put $\tilde{\Phi}(q)=\tilde{\Phi}(-\nabla_yc(x,y_t))=\Phi(x)$, so that for $X\in T_xM$
$$
D\Phi(\bar{x})X= D\tilde{\Phi}(\bar{q_t})(-A_{\bar{x}} X)
$$
For $x=\bar{x}$ and $\eta\in T_{y_t}N$
$$
D\Phi(\bar{x})\hat{\eta} = D\tilde{\Phi}(\bar{q_t})(-A_{\bar{x}} \hat{\eta}) =  D\tilde{\Phi}(\bar{q_t})\eta 
$$
 We put $q_t:=-\nabla_yc(x^t,y_t)$ with $x^t \in \partial^c \psi(y_t)$ and recall $\bar{q_t}=-\nabla_yc(\bar{x},y_t)$.We get $\nabla \psi(y_t) = \nabla_yc(x^t,y_t) = - q_t$ and therefore get $\zeta = q_t - \bar{q_t}$. Using the c-convexity of $D$ to differentiate $c$ at $(x_t, \cexp_{x_t}(\bar{p_t}+s\zeta))$, we get
$$
h''(t) = \tilde{\Phi}(\bar{q_t})+ D\tilde{\Phi}(\bar{q_t})({q_t}-{\bar{q_t}}) = \tilde{\Phi}(q_t) - \int_0^1 D^2_q \tilde{\Phi}(\bar{q_t}+s\zeta)(\zeta,\zeta)(1-s)\d s
$$
We have
$$
\tilde{\Phi}(q_t) = \Phi(x^t) = \Big( \nabla^2_{yy}c(x^t,y_t) - \nabla^2 \psi(y_t)\Big)(\hat{\eta},\hat{\eta})
$$
Using the change of variable $q= -\nabla_yc(x,y_t)\in T_{y_t}M$ (or  equivalently $x =\cexp_{y_t}(q)$), we get
$$
\tilde{\Phi}(q) = \Big(\nabla^2_{yy}c(\cexp_{y_t}(q),y_t) - \nabla^2 \psi(y_t))\Big) (\hat{\eta},\hat{\eta})
$$
Since $\nabla^2 \psi(y_t)$ does not depend on $q$, one gets by the definition of the MTW tensor, for any $\zeta \in T_q(T_{y_t}N)= T_{y_t}N$:
$$
D^2_{q} \tilde{\Phi}(q) (\zeta,\zeta) 
= \frac{\partial^2}{\partial q_{\zeta}^2} \frac{\partial^2}{\partial y_{\hat{\eta}}^2}\big(c(\cexp_{y_t}(q),y_t))\big)
=-\frac{2}{3}\mathfrak{S}_c(\cexp_{y_t}(q),y_t))\ (\bar{\zeta},\hat{\eta}) 
$$
where we put $\bar{\zeta} := -A_{x_s}^{-1} \zeta$ so as to have $\tilde{\bar{\zeta}} = \zeta$.
\end{proof}

\begin{lemma}
\label{lem:ODE}
Let $y \in \Class^1([0,1],\R)$ satisfying for $C > 0$,
\[ 
\begin{cases}
y'(t) \geq -C |y(t)| \\
y(0) = 0
\end{cases}
\]
then $y(t) \geq 0$ for any $t \in [0,1]$.
\end{lemma}
\begin{proof}
First remark that there exists $g \in \Class^0([0,1], \R_+)$ such that $y$ is solution of
\[ 
\begin{cases}
y'(t) = -C |y(t)| + g(t) \\
y(0) = 0
\end{cases}
\]
Assume that $y \leq 0$ on $[0,1]$ then $|y| = -y$ and $y$ satisfies
\[ 
\begin{cases}
y'(t) = C y(t) + g(t) \\
y(0) = 0
\end{cases}
\]
yet the unique solution to this system is $t \mapsto \int_0^t g(s) e^{C(t-s)}ds \geq 0$ which gives $y = 0$.
Now assume that there exists $t_0$ such that $y(t_0) := y_0 > 0$, then the system may be rewritten
\[ 
\begin{cases}
y'(t) = -C |y(t)| + g(t) \\
y(t_0) = y_0
\end{cases}
\]
and has for unique solution $t \mapsto y_0 e^{C (t_0-t)} + \int_{t_0}^t g(s)e^{C(s - t)}ds \geq 0$ on $[t_0, 1]$.
To conclude let us consider $t_* = \inf\{t, y(t) > 0\}$, then we have $y(t) \leq 0$ on $[0, t_*]$ which implies $y = 0$ on $[0, t_*]$ as we have seen previously.
\end{proof}

\begin{proposition}Under hypothesis of Theorem~\ref{th:criterionstrongcconc},
\label{prop:ineqhsec}
\[ h''(t) \geq -C h'(t) + \lambda \|\hat{\eta}\|^2 \]
\end{proposition}
\begin{proof}
We have 
$$
|h'(t)| =  |\sca{\zeta}{\hat{\eta}}|.
$$
We also have 
$$
h''(t) = \Big(\D^2_{yy}c(x^t,y_t) - D^2\psi(y_t)\Big) (\hat{\eta},\hat{\eta}) 
+ \frac{2}{3} \int_0^1 \mathfrak{S}_c( x_s, y_t) (\bar{\zeta}, \hat{\eta}) (1-s) \d s,
$$
where $x_s=\cexp_{y_t}(\bar{q_t}+s\zeta)$. By hypothesis we have 
\[
\Big(\D^2_{yy}c(x^t,y_t) - D^2\psi(y_t)\Big) (\hat{\eta},\hat{\eta}) \geq \lambda \nr{\hat{\eta}}^2
\]
and \eqref{eq:MTWw} gives
\[
\mathfrak{S}_c( x_s,y_t) (\bar{\zeta}, \hat{\eta}) 
\geq - C |\sca{\nabla_{xy}^2c( x_s,y_t)\hat{\eta}}{\bar{\zeta}} | \|\hat{\eta}\|\|\bar{\zeta}\|
\]
The norms $ \|\hat{\eta}\|$ and $\|\bar{\zeta}\|$ can be integrated in the constant by compactness, so we get
\[
\mathfrak{S}_c( x_s,y_t)  (\bar{\zeta}, \hat{\eta})
\geq - C |\sca{\nabla_{xy}^2c( x_s,y_t))\hat{\eta}}{\bar{\zeta}}|
=  -C |\sca{^tA_{x_s}\hat{\eta}}{\bar{\zeta}}| 
\]
Recall that $\bar{\zeta}=-\ A_{x_s}^{-1} \zeta$. Therefore we get
$$
 |\sca{^tA_{x_s}\hat{\eta}}{\bar{\zeta}}| = |\sca{{}^tA_{x_s} \hat{\eta}}{A_{x_s}^{-1} \zeta}| =  |\sca{\zeta}{\hat{\eta}}| =  |h'(t)|,
 $$
We thus have $h''(t) \geq -C |h'(t)| + \lambda \|\hat{\eta}\|^2 $.
Note that $\zeta |_{t = 0} = 0$ so $h'(0) = 0$. Then we can apply Lemma~\ref{lem:ODE} to $h'$, which gives $h'(t) \geq 0$, so we can drop the absolute value and we obtain $h''(t) \geq - C h'(t) +  \lambda \|\hat{\eta}\|^2$.
\end{proof}

\begin{proof}[Proof of Theorem~\ref{th:criterionstrongcconc}]
By compactness we have
\[ C_1 := \inf_{(x,y) \in D, u \in T_xM, \nr{u} = 1} \nr{\nabla_{xy}^2 c(x,y)^{-1} u}^2 > 0  \]
and
\[ C_2 := \inf_{x \in \X, y,z \in \Y^x} \frac{\nr{\nabla_xc(x,y) - \nabla_xc(x,z)}^2}{\d_N(y,z)^2} > 0 \]
such that $\nr{\hat{\eta}}^2 \geq C_1 C_2 \d_N(y,\bar{y})^2$.
By Proposition~\ref{prop:ineqhsec},  we get
\[ h''(t) \geq - C h'(t) + \lambda C_1 C_2 \d_N(y,\bar{y})^2 \]
Using Gr\"onwall's Lemma we then have that $h'(t) \geq g(t)$ with $g$ solution of
\[
\begin{cases}
g'(t) = - C g(t) + \lambda C_1 C_2 \d_N(y,\bar{y})^2 \\
g(0) = 0
\end{cases}
\]
which immediatly gives $g(t) = \left(\frac{\lambda C_1 C_2}{C} \d_N(y,\bar{y})^2 \right)(1 - e^{-Ct})$, so finally we have for $t \in [0,1]$, $h'(t) \geq \left(\frac{\lambda C_1 C_2}{C} \d_N(y,\bar{y})^2 \right)(1 - e^{-Ct})$,
and the by integrating for $t \in [0,1]$ we conclude that
\[ \int_0^1 h'(t) dt \geq {\lambda C_1 C_2}e^{-C}\d_N(y,\bar{y})^2 \]
which is exactly what we wanted in Equation~\eqref{eq:hcconc}.
\end{proof}

\begin{proof}[Proof of Corollary~\ref{cor:strongcconc}]
We want to show that under the hypothesis of Corollary~\ref{cor:strongcconc}, 
we have
$$
\forall y \in \Y\ \forall x \in \partial^c \psi(y),\ D^2_{yy}c(x,y) - D^2\psi(y) \geq \lambda Id,
 $$

We recall that $T:\X \to \Y$ is of class $\Class^1$.
Let $x \in \X$, we first assume that $T(x)\in \inter(\Y)$.
Since $T(x)$ minimizes $c(x,\cdot) - \psi(\cdot)$ we have
\begin{equation}
\label{eq:grad_zero}
\nabla_y c(x,T(x)) - \nabla \psi(T(x)) = 0
\end{equation}
and
\begin{equation}
\label{eq:diff_pos}
D^2_{yy} c(x,T(x)) - D^2 \psi(T(x)) \geq 0
\end{equation}
By differentiating \eqref{eq:grad_zero} with respect to $x$, we get
\begin{equation}
\label{eq:diff_non_singular}
\left( D^2_{yy} c(x,T(x)) - D^2 \psi(T(x)) \right) \circ DT(x) = - D^2_{xy} c(x,T(x)).
\end{equation}
By (STwist) assumption, $ D^2_{xy} c(x,T(x))$ is nonsingular, which implies that $D^2_{yy} c(x,T(x)) - D^2 \psi(T(x))$ is also nonsingular. Since we also know that it is positive semi-definite from \eqref{eq:diff_pos} we get that 
$$
D^2_{yy} c(x,T(x)) - D^2 \psi(T(x)) > 0.
$$ 
We now need to extend this inequality for any $T(x) \in \partial \Y$, including the boundary. By continuity, since $\psi$ is $\Class^2$ on $\Y$, $c$ is $\Class^2$ on $D$ and $T$ is $\Class^1$ on $\X$,  Equations~\eqref{eq:diff_pos} and \eqref{eq:diff_non_singular} still hold when $T(x) \in \partial \Y$. Moreover (STwist) being satisfied on $D$,  we have $D^2_{yy} c(x,T(x)) - D^2 \psi(T(x)) > 0$ for any $x \in \X$. By compactness of $\X$, there exists $\lambda > 0$ such that 
$$
\forall x\in \X\quad   D^2_{yy} c(x,T(x)) - D^2 \psi(T(x)) \geq \lambda Id.
$$
We conclude using that $T(x) = y$ is equivalent to $x \in \partial^c \psi(y)$. 
\end{proof}

\section{Stability of optimal transport map for MTW cost}\label{sec:otstability}
In this section, we show that the stability results of Section~\ref{sec:stability} can be applied to optimal transport maps. We consider two compact Riemannian manifolds $M$ and $N$ in $\R^d$ and still denote by $d_N$ the distance on $N$.
\begin{theorem}[Stability in optimal transport]
\label{th:OTstability}
Let $\mu \in \Prob(M)$ and $\nu \in \Prob(N)$ be two probability measures.
Let $c : M \times N \to \R$ be a cost function of class $\Class^4$  that  satisfies \emph{(STwist)} and \eqref{eq:MTWw} hypothesis.
Let $T : M \to N$ be an optimal transport map between $\mu$ and $\nu$ of class $\Class^1$ for the cost $c$ and assume that its associated Kantorovich potential $\psi : M \to \R$ is of class $\Class^2$.
\begin{itemize}
\item Let $\tilde{\nu} \in \Prob(N)$ be any probability measure, and $S : M \to N$ be an optimal transport map between $\mu$ and $\tilde{\nu}$.
Then we have
\[ \nr{\d_N(T, S)}^2_{L^2(\mu)} \leq C W_1(\nu, \tilde{\nu})\]
where $W_1$ denotes the $1$-Wasserstein distance and $C$ is a constant depending on the cost $c$, $M$ and $N$.
\item Let $\tilde{\mu} \in \Prob(M)$, $\tilde{\nu} \in \Prob(N)$ and $\tilde{\gamma}$ be an optimal transport plan between $\tilde{\mu}$ and $\tilde{\nu}$. Then we have
\[  \Wass_1(\tilde{\gamma}, \gamma_T) \leq C \left( \Wass_1(\tilde{\mu},\mu) + \Wass_1(\nu,\tilde{\nu}) \right)^{1/2} \]
where $\gamma_T = (Id,T)_\# \mu$  and $C$ is a constant depending on the cost $c$, $M$ and $N$.
\end{itemize}
\end{theorem}
\begin{proof}
Since $M$ and $N$ are compact we have strong duality with a cost $c$ that is Lipschitz on $M \times N$ so $S$ is induced by a Lipschitz potential.
Since $T \in \Class^1$, $\psi \in \Class^2$ and $c \in \Class^4$ satisfies (STwist) and \eqref{eq:MTWw}, we can then apply Corollary~\ref{cor:strongcconc} to $\psi$, which gives that it is strongly c-concave on $N$, with modulus $\omega(\d_N(y,z)) = C \d_N(y,z)^2$.
Then the first result is given by Theorem~\ref{th:stability-cconc} and the second is given by Proposition~\ref{prop:stabbothmeasureW1}.
\end{proof}

For simplicity, the above theorem is stated in a restrictive way as it requires $c$ to be smooth on the whole product space $M \times N$. It may happen that the regularity conditions such as (STwist) and \eqref{eq:MTWw} are not satisfied on the whole product space $M\times N$, but only on a subset $D \subseteq M\times N$. In this case we can still obtain stability with respect to the target measure if we can show that optimal transports plans are supported on this subset $D$. This is treated independently on  examples of Sections~\ref{sec:reflector} and~\ref{sec:gaussmeasure}.

\section{Stability for the reflector cost on the sphere}\label{sec:reflector}
In this section, we apply a stability result of Section~\ref{sec:stability} to the reflector antenna problem. It is known that this problem amounts to solving an optimal transport problem on the unit sphere $M = N = \Sph^{d-1}$ for the cost function $c(x,y) = - \ln(1 - \sca{x}{y})$~\cite{wang2004design}, extended by $+\infty$ on the diagonal $\{x=y\}$. One of the key element in the proof is to show that optimal transport maps are supported on compact sets that avoid the diagonal 
\begin{equation}
\label{eq:Deps}
 D_\eps = \{ (x,y) \in M^2 \mid \d_M(x,y) \geq \eps\} 
 \end{equation}
 where $d_M$ is the geodesic distance on $M$. 
We first need the following definition. 
\begin{definition}
\label{def:massconcentration}
  Given a probability measures $\mu\in\Prob(M)$, we
  put $$M_{\mu}(r) = \sup_{x\in M}
  \mu(\B(x,r)).$$
\end{definition}

\begin{theorem}
\label{th:stabilitymapsreflector}
Let $c(x,y) = -\ln(1 - \sca{x}{y})$ be  the reflector cost  on the sphere $M=\Sph^{d-1}$. Let $\mu, \nu_0, \nu_1 \in \Prob(M)$ be such that $\mu$ and $\nu_0$ are absolutely continuous with respect to the Lebesgue measure with strictly positive $\Class^{1,1}$ densities. Let $T_i$ be optimal transport maps between $\mu$ and $\nu_i$.  Then for all $\beta>0$, there exists a constant $C>0$ depending on $\mu,\nu_0$ and $\beta$ such that 
\[ \forall \nu_1\in \Prob(N) \hbox{ s.t. } M_{\nu_1}(\beta) < 1/8, \quad \nr{ \d_M(T_0,T_1)}^2_{L^2(\mu)} \leq C\ W_1(\nu_0, \nu_1) \]
where $d_M$ is the geodesic distance on $M$. 
\end{theorem}

%

The main difficulty to prove the previous theorem is to show that the optimal transport plan is supported on the compact set $D_\eps$ for some $\epsilon$. This is done in the following subsection in a more general setting. 

\subsection{Support of the optimal transport plan} In this subsection, we show that optimal transport plans are supported on compact sets of the form $D_\epsilon$. Since our result holds in  a slightly more general context than the sphere,  we consider that  $M$ can be any smooth complete Riemannian manifold. 
Let  $c: M \times M \to \R$ be any cost bounded from  below that satisfies $c(x,y) = h(\d_M(x,y))$ where $h : \R_+ \to \R$ is a continuous decreasing function such that $h(0) = +\infty$ and $h(t) < + \infty$ for $t > 0$.

\begin{theorem}\label{th:kantorovichreflector}
  Let $\mu,\nu \in \Prob(M)$ and $\beta>0$ such that both $M_\mu(\beta) <
  1/8$ and $M_\nu(\beta) <
  1/8$, then there exists a constant $\eps>0$ such that any optimal
  transport plan $\gamma\in\Gamma(\mu,\nu)$ is concentrated on
  $D_\eps$.
\end{theorem}

Similar results have already been obtained in different settings~\cite{gangbo2007existence,buttazzo2018continuity,loeper2011regularity}, but none of them can be applied to discrete measures and therefore does not imply our result. 
W. Gangbo and V.Oliker~\cite{gangbo2007existence} work with Borel measures that vanish on $(d-1)$-rectifiable sets.  G. Buttazzo et al.~\cite{buttazzo2018continuity} consider  multimarginal optimal transport problems for constant measures.  G.Loeper~\cite{loeper2011regularity} considers two measures $\mu$ and $\nu$ such that $\mu \geq m \emph{dVol}$ with $m > 0$ and $\nu$ satisfies for any $\eps \geq 0$ and $x \in M$, $\nu(\B(x,\eps)) \leq f(\eps) \eps^{n(1 - 1/n)}$ for some function $f : \R_+ \to \R_+$ satisfying $\lim_{t \to 0} f(t) = 0$. These hypothesis imply that neither $\mu$ nor $\nu$ can be discrete.

Our proof is an adaptation of their proofs in  a different context. Lemma~\ref{lem:ccycl} is inspired by~\cite{gangbo2007existence} while Lemma~\ref{lem:pairs} and the overall strategy of the proof come from~\cite{buttazzo2018continuity}. The main difference is that here we work on any measure satisfying $M(\beta) < 1/8$, including discrete measures, which is useful for semi-discrete optimal transport.

\begin{remark}
Our proof requires $M(\beta) < 1/8$ but we believe that the theoretical bound is $M(\beta) \leq 1/2$, which is enough to guarantee that there exists a transport plan with finite global cost, as showed in the following lemma.
It is easy to show that we cannot expect a greater bound.  Take for example $x \neq y$ in $\Sph^{d-1}$, $\eps \in ]0, 1/2[$, $\mu = 1/2(\delta_x + \delta_y)$ and $\nu = (1/2 + \eps) \delta_x + (1/2 - \eps)\delta_y $. Any transport plan between $\mu$ and $\nu$ will send a set of measure at least $\eps$ from $x$ to itself for which the cost is infinite.
\end{remark}

The end of this section is mainly dedicated to the proof of Theorem~\ref{th:kantorovichreflector}, which is necessary to guarantee that the optimal transport plan is supported where the cost is regular enough and the MTW tensor is non-negative, and allow us to apply our strong c-concavity result in order to obtain the stability results of Theorem~\ref{th:stabilitymapsreflector}.
We first show in the following lemma that there exists a transport plan with bounded total cost.
\begin{lemma} \label{lem:finite-cost}
  If $M_\mu(\beta) \leq 1/2$ and $M_\nu(\beta) \leq 1/2$ for some $\beta>0$, then there exists
  $\gamma\in\Gamma(\mu,\nu)$ s.t.
  $$ \int c\dd\gamma \leq h(\beta/2).$$
\end{lemma}


The proof of Lemma~\ref{lem:finite-cost} relies on the following result, which can be seen as a continuous formulation of Hall's mariage Lemma. A proof is given in~\cite[Theorem 1.27]{villani2003topics}. 

\begin{lemma}[Continuous Hall's  marriage lemma]\label{lem:hall}
Let $M,N$ be Polish spaces, and let $P$ be a closed subset of $M\times
N$. Given $\mu \in \Prob(M)$ and $\nu\in\Prob(N)$, the following
propositions are equivalent:
\begin{itemize}
\item[(i)] $\exists \gamma \in\Gamma(\mu,\nu)$ such that $\spt(\gamma)\subseteq P$ ;
\item[(ii)] for every Borel subset $B\subseteq M$, 
$$\nu(\{ y\in N \mid \exists x\in B \hbox{ s.t. } (x,y)\in P) \})\geq \mu(B). $$ 
\end{itemize}
\end{lemma}

\begin{proof}[Proof of Lemma~\ref{lem:finite-cost}] We are going to apply the Continuous Hall's marriage lemma to the set $P=\{(x,y)\in M\times N,\ \d_M(x,y) \geq \beta/2\}$. Let $B$ be any Borel set of $X$. 
We first assume that the diameter of $B$ is a most $\beta$ so that $B\subseteq \B(x_0,\beta)$ for some  $x_0 \in B$.
   Then, $\mu(B) \leq \mu(\B(x_0,\beta)) \leq 1/2 $ using $M_\mu(\beta) \leq 1/2$. having also $M_\nu(\beta) \leq 1/2$ we get
  \begin{align*}
    \nu(\{ y \in N \mid \exists x \in B,~ \d_M(x,y) \geq  \beta/2 \})
    &\quad \geq \nu(\{ y \in N \mid  \d_M(x_0,y) \geq  \beta/2 \}) \\
    &\quad = 1 -\nu(\B(x_0,\beta/2)\\
    &\quad  \geq 1/2\\
    &\quad  \geq \mu(B).
  \end{align*}
 Assume now that the diameter of $B$ is greater than $\beta$. 
 Then  there exist $ x,x' \in B$ such that $\d_M(x,x')\geq \beta$ and  the left hand side of the previous inequation is equal to 1 and the condition is obviously satisfied. 
 We can therefore apply Lemma~\ref{lem:hall}, which  implies the 
  existence of a transport plan $\gamma$ between $\mu$ and $\nu$ such
  that for any pair $(x,y)\in\spt(\gamma)$ one has $\d_M(x,y)\geq
  \beta/2$. Since $h$ is decreasing, we have $c(x,y)\leq h(\beta/2)$ for every pair $(x,y)\in\spt(\gamma)$,   which implies the desired result.
  \end{proof}

\begin{lemma} \label{lem:pairs}
  Let $\gamma$ be an optimal transport map between $\mu$ and $\nu$ for the cost $c$
  and let $\beta>0$ such that $M_\mu(\beta) < 1/8$ and $M_\nu(\beta) < 1/8$. Then, for any
  optimal transport plan $\gamma\in\Gamma(\mu,\nu)$, there exists
  pairs $(x_0,y_0), (x_0',y_0')\in\spt(\gamma)$ such that the four
  points $x_0,y_0,x_0',y_0'$ are at distance at least $\min(\eps, \beta)$ with $\eps :=
  h^{-1}(4h(\beta/2))$ from each other.
\end{lemma}

\begin{proof}
  Since $\gamma$ is an optimal transport plan, its cost is less than
  the cost of the transport plan constructed in
  Lemma~\ref{lem:finite-cost}. 
  Since $h$ is deacreasing and by definition of $\Delta_\eps$, we have for any $\eps > 0$,
  \[ h(\eps) \gamma(\Delta_\eps) \leq \int_{\Delta_\eps} c \dd \gamma \leq h(\beta/2) \]
  Note that we can consider $h(\beta/2) > 0$, choosing a smaller $\beta$ if necessary. 
  Then if $\eps = h^{-1}(4h(\beta/2))$ we get 
  $\gamma(\Delta_\eps) \leq \frac{1}{4}, $ thus
  proving the existence of a pair $(x_0,y_0)\in\spt(\gamma) \setminus
  \Delta_\eps$. 

  Since $M_\mu(\beta) < 1/8$, one has
  $$\gamma((\B(x_0,\beta) \cup \B(y_0,\beta)) \times \Sph^{d-1}) \leq \mu(\B(x_0,\beta)) + 
  \mu(\B(y_0,\beta))< \frac{1}{4}, $$
  Similarly, $M_\nu(\beta) < 1/8$, gives
  $$\gamma(\Sph^{d-1} \times (\B(x_0,\beta) \cup \B(y_0,\beta))) \leq
  \nu(\B(x_0,\beta)) + \nu(\B(y_0,\beta) <
  \frac{1}{4}, $$
  so that
  \begin{align*}
    &\gamma(\{ (x,y) \in M^2 \mid \d_M(x, x_0) > \beta,~~ \d_M(y, y_0)  > \beta,~ \d_M(y, - x_0) > \beta, \\
&\phantom{\gamma(\{ (x,y) \in M^2 \mid} \d_M(x, y_0) > \beta  \hbox{ and }  \d_M(x, y) > \eps \} \\
    &\quad = \gamma\bigg(M^2 \setminus \bigg[ (\B(x_0, \beta) \cup \B(y_0, \beta)) \times \Sph^{d-1} \\
    &\phantom{\quad = \gamma\bigg(M^2 \setminus} \cup \Sph^{d-1} \times (\B(x_0, \beta) \cup \B(y_0, \beta)) \cup \Delta_\eps\bigg] \bigg) \\
    &\quad \geq 1 - \gamma((\B(x_0,\beta) \cup \B(y_0,\beta)) \times \Sph^{d-1})\\
    &\qquad\quad- \gamma(\Sph^{d-1} \times (\B(x_0,\beta) \cup \B(y_0,\beta))) - 
    \gamma(\Delta_\eps) 
    \quad> 1/4.
  \end{align*} This proves the existence of
  $(x'_0,y_0')\in\spt(\gamma)$ such that $\d_M(x_0, x_0') > \beta$ and
  $\d_M(y_0, y_0') > \beta$ and $\d_M(x'_0, y'_0)\geq \eps$ and allows us to conclude.
\end{proof}

\begin{lemma} \label{lem:ccycl}
  Assume that $c$ is bounded from below by a constant $c_{min}$.
  Let $S\subseteq M \times M$ be a $c$-cyclically monotone set , which contains two
  pairs $(x_0,y_0), (x'_0, y'_0)$ such that the pairwise distance
  between the points $x_0,y_0,x'_0,y'_0$ is at least $\eps>0$. Then,
  $$ \forall (x,y) \in S,\quad c(x,y) \leq C_\eps := h(\eps) + 2h(\eps/2) + 2 |c_{min}|. $$
\end{lemma}

\begin{proof}
  Using the $c$-cyclical
  monotonicity of $S$ and $c\geq c_{min}$ one has
  \[ c(x,y) \leq c(x,y) + c(x_0,y_0) + c(x'_0,y'_0) - 2 c_{min} \leq F(x,y) + 2|c_{min}| \]
  where
  $$\left\{\begin{aligned}
    &F(x,y) = \min(c(x,y_0) + R_1(y), c(x,y'_0) + R_2(y)) \\
    &R_1(y) = \min(c(x_0,y) + c(x'_0,y'_0), c(x_0,y'_0) + c(x'_0,y))) \\
    &R_2(y) = \min(c(x_0,y) + c(x'_0,y_0), c(x_0,y_0) + c(x'_0,y))).
  \end{aligned}\right.$$
By assumption, we have $\d_M(x_0, x_0') \geq \eps$,
 thus  $\max(\d_M(x_0,y), \d_M(x_0,y) \geq \eps /2$.
  Then, since $h$ is decreasing, one has  $\min(c(x_0,y),c(x'_0,y)) \leq h(\eps/2)$.
  We also have $c(x'_0,y'_0) \leq h(\eps)$ and $c(x_0,y'_0) \leq h(\eps)$, which leaves us with
  $$R_1(y) \leq h(\eps) + \min(c(x_0,y), c(x'_0,y)) \leq h(\eps) + h(\eps/2), $$
  and the same bound holds for $R_2(y)$. Using the same argument we 
  get $\min(c(x,y_0),c(x,y_0')) \leq h(\eps/2)$ and thus,
  \begin{equation*}
    F(x,y) \leq h(\eps) + h(\eps/2) + \min(c(x,y_0), c(x,y'_0)) \leq h(\eps) + 2h(\eps/2). \qedhere
  \end{equation*}
\end{proof}

\noindent \textbf{Proof of Theorem~\ref{th:kantorovichreflector}.}   
Let $\beta>0$ such that $M(\beta)> 1/8$. Let $\gamma$ be an optimal transport plan, and denote by  $S$  its
  support. By Lemma~\ref{lem:finite-cost}, the cost of this transport
  plan is finite. This implies that $S$ is $c$-cyclically monotone. Recall that by assumption, the cost $c$ is bounded from below. Therefore by 
    Lemmas~\ref{lem:pairs} and \ref{lem:ccycl} one has 
  $$ \forall (x,y) \in S,\quad c(x,y) \leq C_\eps := h(\eps) + 2h(\eps/2) + 2 |c_{min}|. $$
 where $\eps = \min(\beta,h^{-1}(4h(\beta/2)))$.   
  This directly implies that $S \subseteq D_\delta$ with $\delta = h^{-1}(C_\eps)$.

\subsection{Proof of Theorems~\ref{th:stabilitymapsreflector}}
Here, we come back to the sphere case, i.e. $M = \Sph^{d-1}$.
We recall that the reflector cost is given on  $M^2$ by $c(x,y) = - \ln(1 - \sca{x}{y})$. Note that on the unit sphere, $\d_M(x,y) = \arccos(\sca{x}{y})$, hence the reflector cost is of the form  $c(x,y) = h(\d_M(x,y))$ with $h(t) = -\ln(1 - \cos(t))$ and satisfies the assumptions of Theorem~\ref{th:kantorovichreflector}. 

\begin{lemma}
\label{lem:Depscconvex}
For $\eps < 2$, $D_\eps$ is symmetrically $c$-convex.
\end{lemma}
\begin{proof}
A simple computation gives for $x \in M$, that $\nabla_x c(x,\cdot) : M \setminus \{x\} \to T_x M$ is one to one and given  by
\[ \nabla_x c (x,y) = \frac{y - \sca{x}{y}x}{1 - \sca{x}{y}} \]
and the inverse of $-\nabla_xc(x, \cdot)$ is
\[ \cexp_x(p) = \left( 1 - \frac{2}{1 + \nr{p}^2} \right) x - \frac{2}{1 + \nr{p}^2} p \]
Let $(x,y_0)$ and $(x, y_1)$ in $D_\eps$, and define the $c$-segment $(y_t) = [y_0, y_1]_x$.
For $p_0 = \nabla_x c(x,y_0)$ and $p_1 = \nabla_x c(x,y_1)$, we put $p_t = (1-t) p_0 + t p_1$, so that $y_t = \cexp_x(p_t)$.
We want to show that $(x,y_t) \in D_\eps$, hence we only have to show that $\d_M(x,y_t) \geq \eps$ .
We have
\[ x - y_t = \frac{2}{1 + \nr{p_t}^2} x + \frac{2}{1 + \nr{p_t}^2} p_t.\]
Since $x$ is orthogonal to $p_t$ and $\nr{x} = 1$, we get 
\[ \d_M(x,y_t) = \arccos(\sca{x}{y_t}) = \arccos\left(1 - \frac{2}{1 + \nr{p_t}^2}\right).\]
So $\d_M(x,y_t) \geq \eps$ is satisfied if $1 - \frac{2}{1 + \nr{p_t}^2} \leq \cos(\eps)$. Since $cos(\eps) \geq 1 - \eps^2/2$ it is sufficient to show that 
\[ \frac{2}{1 + \nr{p_t}^2} \geq \eps^2 /2. \]
%
Since $\nr{p_t} \leq \max(\nr{p_0}, \nr{p_1})$, and by symmetry of $p_0$ and $p_1$ it is sufficient to show that $\nr{p_0}^2 \leq \frac{4}{\varepsilon^2} - 1$.
Again using that $\nr{x} = \nr{y_0} = 1$, we have
\[ \nr{p_0}^2 = \nr{\frac{y_0 - \sca{x}{y_0}x}{1 - \sca{x}{y_0}^2}}^2 = \frac{1 - \sca{x}{y_0}^2}{(1 - \sca{x}{y_0})^2} = \frac{1 + \sca{x}{y_0}}{1 - \sca{x}{y_0}}\]
Finally using the relation $\sca{x}{y_0} = 1 -  \nr{x-y_0}^2 / 2$, we get
\[ \nr{p_0}^2 = \frac{4}{\nr{x-y_0}^2} - 1 \leq \frac{4}{\varepsilon^2} - 1 \]
and in conclusion, $D_\eps$ is $c$-convex.
Note that by symmetry it is obviously symmetrically c-convex.
\end{proof}

\noindent \textbf{End of proof of Theorem~\ref{th:stabilitymapsreflector}}
Since $\mu$ and $\nu_0$ are absolutely continuous there exists $\beta > 0$ such that $M_\mu(\beta) < 1/8$, $M_{\nu_0}(\beta) < 1/8$ and $M_{\nu_1}(\beta) < 1/8$.
Therefore, by Theorem~\ref{th:kantorovichreflector}, there exists $\eps >0$ such that for every $x \in M$, $(x, T_i(x)) \in D_\eps$. 
The set $D_\eps$ is a compact set and symmetrically c-convex by Lemma~\ref{lem:Depscconvex}. Recall that the optimal transport map $T_0$ between $\mu$ and $\nu_0$ is of the form $T_0(x)=\argmin_{y\in N} c(x,y) - \psi_0(y)$, where $\psi_0:N\to \Rsp$ is a $c$-concave function. 
Since $\mu$ and $\nu_0$ have $\Class^{1,1}$ strictly positive densities, a result of Gregoire Loeper~\cite[Theorem 2.5]{loeper2011regularity} implies that $\psi_0$ is of class $\Class^3$ and that $T:x\mapsto \cexp_x(\nabla \psi^c(x))$ is of class $\Class^2$. As seen in the proof of Theorem~\ref{th:kantorovichreflector}, $\psi_1$ is $c$-concave for the truncated cost, which is Lipschitz, and is therefore also Lipschitz.  Furthermore, it is known that the reflector cost satisfies MTW and (STwist)~\cite{loeper2011regularity}. We  can thus apply Corollary~\ref{cor:strongcconc} which gives that $\psi_0$ is strongly c-concave on $D_\eps$. We then conclude by applying Theorem~\ref{th:stability-cconc}.


\section{Prescription of Gauss curvature measure}\label{sec:gaussmeasure}
The problem of Gauss curvature measure prescription for a convex body
has been introduced by A.D. Aleksandrov in
1950~\cite{alexandrov1950convex} and has been shown to be equivalent
to an optimal transport problem on the
sphere~\cite{oliker2007embedding,bertrand2016prescription}. In this
section we apply our stability result to this  optimal
transport problem.

To this purpose we define the Gauss curvature measure introduced
in~\cite{alexandrov1950convex}.  Let $K \subseteq \R^d$ be a closed
bounded convex body such that $0 \in \inter(K)$.  We denote by
$\rho_K: \Sph^{d-1} \to \R$ the radial parametrization of $\partial K$
defined for any direction $x$ in the sphere $\Sph^{d-1}$ by $\rho_K(x)
= \sup \{ r \in \R \mid rx \in K\}$. This induces a  homeomorphism
$\overrightarrow{\rho_K}$ from $\Sph^{d-1}$ to $\partial K$ defined by
\begin{align*}
\overrightarrow{\rho_K} : \Sph^{d-1} &\to \partial K \\
x &\mapsto \rho_K(x)x
\end{align*} 
We call (multivalued) Gauss map, the map $\mathcal{G}_K$ which maps a point $x\in\partial K$ to the set of unit exterior normals to $K$ at $x$, namely
$$G_K(x) = \{ n \in \Sph^{d-1} \mid x\in\arg\max_K\sca{n}{\cdot} \}.  $$
Note that  $\mathcal{G}_K(x)$ is a set when $K$ is not smooth at $x$.
Through this section, we denote by $\sigma$ the uniform probability measure on the sphere $\Sph^{d-1}$, i.e. the normalized $(d-1)$-dimensional Hausdorff measure.
\begin{definition}[Gauss curvature measure]
  Let $K$ be a bounded convex body containing $0$ in its interior. The
  \emph{Gauss curvature measure} of $K$, denoted $\mu_K$, is a
  probability measure over $\Sph^{d-1}$ defined for any Borel subset $A \subseteq
  \Sph^{d-1}$ by $\mu_K(A) = \sigma(\mathcal{G}_K \circ
  \overrightarrow{\rho_K}(A))$.
\end{definition}

The \emph{Gauss curvature measure prescription problem} is the
following inverse problem: given a measure $\mu \in
\Prob(\Sph^{d-1})$, is it possible to find a convex body $K$ such that
$\mu=\mu_K$ ?  It is well-known that convexity of $K$ implies that for
every non-empty spherical convex subset $\Theta \subsetneq \Sph^{d-1}$ -- i.e. subsets $\Theta$ that contains any mimimizing geodesic between any pair of its points ---
we have
\begin{equation}
\label{eq:convexmeasure}
\mu_K(\Theta) < \sigma(\Theta_{\pi/2})
\end{equation}
with $\Theta_{\pi/2} = \{ x \in \Sph^{d-1} \mid d_M(x,\Theta) < \pi/2 \}$, and where where $d_M$ is the geodesic distance on the sphere. 
Aleksandrov's theorem states that Equation~\eqref{eq:convexmeasure} is in fact a sufficient condition for $\mu$ to be the Gauss curvature measure of a convex body.
\begin{theorem}[Aleksandrov]
Let $\mu \in \Prob(\Sph^{d-1})$ be a probability measure satisfying condition~\eqref{eq:convexmeasure}, then there exists a unique (up to homotheties) convex body $K \subseteq \R^d$ with $0 \in \inter(K)$ such that $\mu$ is the Gaussian curvature measure of $K$.
\end{theorem}

\subsection{An optimal transport problem}
Following~\cite{oliker2007embedding, bertrand2016prescription} we
briefly recall that this inverse problem can be recast as an optimal
transport problem on the sphere for the cost $c(x,n) =
-\ln(\max(0,\sca{x}{n}))$, which takes value $+\infty$ when
$\sca{x}{n}\leq 0$. Let $\mu$ be any measure in $\Prob(\Sph^{d-1})$
satisfying condition~\eqref{eq:convexmeasure}. Note that the very same cost plays an important role in the theory of unbalanced optimal transport \cite{chizat2018interpolating,liero2018optimal,gallouet2021regularity}.

In the following proposition, we use the notion of \emph{support function}
of a convex set $K$, defined by
\[ h_K(n) = \sup_{x \in \Sph^{d-1}} \rho_K(x) \sca{x}{n}.  \]
\begin{proposition}[\cite{oliker2007embedding, bertrand2016prescription}]
Let $\sigma \in \Prob(\Sph^{d-1})$ be the uniform measure over the sphere, let $K$ be a compact convex body containing zero in its interior, and let $\mu = \mu_K$.  Then,
\begin{itemize}
\item The map $T_K:\Sph^{d-1} \to \Sph^{d-1}$ defined  $\sigma$-a.e by 
\[ T_K(n) = (\Grhok)^{-1}(n) \]
 is the optimal transport map between  $\sigma$ and $\mu$ for the cost $c$. \\

\item The functions $\phi_K = - \ln(h_K)$ and $\psi_K = \ln(\rho_K)$ are maximizers of the Kantorovich dual problem. In particular we have
\begin{equation} \label{eq:gauss:kd}
\int_{\Sph^{d-1}} c(T_K(n),n) \d \sigma(n) = 
\int \phi_K(n) \dd \sigma(n) + \int \psi_K(x) \dd \mu_K(x). \end{equation}
\end{itemize}
\end{proposition}
For the sake of completeness, we recall the proof of this proposition. 
\begin{proof} Let $(x,n) \in \Sph^{d-1}\times \Sph^{d-1}$ be such that
  $c(x,n)< +\infty$,
  i.e. $\sca{x}{n} >0$. Then,
  \begin{equation}\label{eq:kd:admiss} h_K(n) = \max_{y\in K}\sca{n}{y} \geq \sca{n}{\rho_K(x)x} =
    \rho_K(x) \sca{n}{x},
  \end{equation}
with equality if and only if $n\in \mathcal{G}_K(x)$.  Since all
quantities are positive, taking the logarithm, we see that $\phi_K(n)
+ \psi_K(x) \leq c(x,n)$, ensuring that $(\phi_K, \psi_K)$ are
admissible for the dual Kantorovich problem.

  Note that e.g. by \cite{bertrand2016prescription} $\sigma$-a.e. direction $n\in\Sph^{-d-1}$ is normal to a unique point in $\partial K$. This implies that the map $T_K = (\Grhok)^{-1}$ is well
  defined $\sigma$-a.e. The equality case of \eqref{eq:kd:admiss} gives
$$ \phi_K(n) + \psi_K(T_K(n)) \leq c(x,T_K(n)).$$ Integrating this
  equality with respect to $\sigma$ directly gives
  \eqref{eq:gauss:kd}. In turn, Kantorovich duality implies that $T_K$
  is an optimal transport between $\sigma$ and $\mu$, and  that $(\phi_K,\psi_K)$ is a
  maximizer in the dual Kantorovich problem.
\end{proof}

\subsection{Stability of transport maps}
In this subsection we apply our stability result to the Gauss curvature measure prescription problem. We introduce the following notation:
$$ \K(r,R) = \{ K \subseteq \Rsp^d\hbox{ convex, compact } \mid B(0,r) \subseteq K \subseteq B(0,R)\}.$$

\begin{proposition}
\label{prop:stabcurvature}
Let $K$ be a strictly convex and $\Class^2$ compact convex body
containing $0$ in its interior. Then, for any $R>r>0$, there exists a constant $C$ depending on $K$, $r$ and $R$ such that
\[
\forall L \in \K(r,R),\quad 
\nr{\d_M(T_K,T_L)}^2_{L^2(\sigma)} \leq C  \Wass_1(\mu_K, \mu_L). \]
\end{proposition}


Note that in addition to the strict convexity and smoothness of $K$, the constant $C$ also depends on the anisotropy of $K$ --- i.e. the radii $R_K\geq r_K > 0$ such that $K\in \K(r_K,R_K)$.
The end of the section is devoted to the proof of
Proposition~\ref{prop:stabcurvature}. We  need to check that
the hypothesis of Corollary~\ref{cor:strongcconc} are satisfied for
the cost $c(x,n) = -\ln(\max(0,\sca{x}{n}))$.

\begin{lemma} \label{lemma:Depscurvature} Given any $R>r>0$, there exists $\eps>0$ such that for any set $K \in\K(r,R)$ and any $c$-optimal transport plan  $\gamma\in\Gamma(\sigma,\mu_K)$, one has
  $$
\spt(\gamma) \subseteq D_\eps,
$$
where $D_\eps = \{ (x,n) \in (\Sph^{d-1})^2 \mid d_M(x,n) \leq \pi/2 - \eps \}.$
\end{lemma}

\begin{proof} 
  By hypothesis,  $r \leq \rho_K(x) \leq R$ for all $x\in\Sph^{d-1}$,
where $\rho_K$ is the radial function of the convex $K$. Since
$$h_K(n) = \sup_{x \in \Sph^{d-1}} \rho_K(x) \sca{x}{n}, $$ we also have $r < h_K(n) < R$. Hence the two Kantorovich potential  $\phi_K(n) = -\ln(h_K(n))$ and $\psi_K(x) = \ln(\rho_K(x))$ therefore satisfy
$$ \phi_K(n) + \psi_K(x)  \leq -\ln(r) + \ln(R) = \ln(R/r),$$
By strong Kantorovich duality $\phi_K(n) + \psi_K(x) = c(x,n)$ on $\spt(\gamma)$, which implies that $c$ is bounded by $\ln(R/r)$ on $\spt(\gamma)$, i.e. for any $(x,n) \in \spt(\gamma)$, one has
$$ c(x,n) = - \ln(\max(0,\sca{x}{n})) \leq \ln(R/r), $$
implying that $\sca{x}{n}\geq r/R$ and $d_M(x,n) = \arccos(\sca{x}{n}) \leq \arccos(r/R)$.
Finally $(x,n) \in D_\eps$ with $\eps = \pi /2 -  \arccos(r/R)$. 
\end{proof}

\begin{lemma}\label{lemma:convexity-gauss}
The set $D_\eps = \{ (x,n) \in (\Sph^{d-1})^2 \mid d_M(x,n) \leq \pi/2 - \eps \}$ is symmetrically c-convex for the cost $c(x,n) = -\ln(\max(0,\sca{x}{n}))$.
\end{lemma}
\begin{proof}
We have
\[\nabla_x c(x,n) = - \frac{n}{\sca{x}{n}} + x \]  
and by inverting $- \nabla_x c(x, \cdot)$ we get
\[\cexp_x(p) = \frac{p + x}{\sqrt{1 + \nr{p}^2}}\]
Let $(x,y_0) \in D_\eps$ and $(x,y_1) \in \D_\eps$ then we have $y_t = \cexp_x(p_t)$ where $p_0 = - \nabla_x c(x,y_0)$ and $p_1 = - \nabla_x c(x,y_1)$ and $p_t = (1-t) p_0 + t p_1$.
By symmetry we can consider that $\nr{p_t} \leq \nr{p_0}$, which implies $\frac{1}{\sqrt{1 + \nr{p_t}^2}} \geq \frac{1}{\sqrt{1 + \nr{p_0}^2}}$ and thus
\begin{align*}
 d_M(x,y_t) &= \arccos(\sca{x}{y_t}) = \arccos\left(\frac{1}{\sqrt{1 + \nr{p_t}^2}}\right) \\
 &\leq \arccos\left(\frac{1}{\sqrt{1 + \nr{p_0}^2}}\right) = d_M(x,y_0) \leq \frac{\pi}{2} - \eps 
 \qedhere
\end{align*}
\end{proof}

\noindent \textit{End of proof of Proposition~\ref{prop:stabcurvature}.} 
The map $T_K$ (resp. $T_L$) is the optimal transport map between the uniform measure $\sigma$ on $\Sph^{d-1}$ and $\mu_K$ (resp. $\mu_L$) for the cost $c(x,n) = -\ln(\max(0,\sca{x}{n}))$.
From Lemma~\ref{lemma:Depscurvature}, for any $n \in \Sph^{d-1}$ we have $(T_K(n),n) \in D_\eps$ and $(T_L(n),n) \in D_\eps$.
Note that for $(x,n) \in D_\eps$, one has $\sca{x}{n}>0$ and therefore $c(x,n) = -\ln(\sca{x}{n}) = -\ln(\cos(\d_M(x,n)))$. It has been shown in~\cite{gallouet2021regularity} that this cost satisfies (STwist) and \eqref{eq:MTWw} on $D_\eps$.
By Lemma~\ref{lemma:convexity-gauss} the set $D_\eps$ is a symmetrically c-convex compact set.

Finally it remains to show that $\psi_K$ is of class $\Class^2$ and $T_K$ is of class $\Class^1$.
Since $\partial K$ is $\Class^2$, its radial parametrization $\rho_K$ is also $\Class^2$, so $\psi_K = \ln(\rho_K)$ of class $\Class^2$.
Furthermore $\overrightarrow{\rho_K}(x) = \rho_K(x) x$ is a $\Class^1$ diffeomorphism.
Since $K$ is stricly convex and $\partial K$ is of class $\Class^2$, its associated Gauss map $\mathcal{G}_K$ is a $\Class^1$ diffeomorphism. 
We thus have that $T_K = (\Grhok)^{-1}$ is of class $\Class^1$.
By Corollary~\ref{cor:strongcconc}, we know that $\psi_K$ is strongly c-concave. We conclude by applying Theorem~\ref{th:stability-cconc}.
\qed


\end{document}